\documentclass{article}
 \usepackage{amsmath}
  \usepackage{amsthm}
 \usepackage{amssymb}

      \usepackage[english]{babel}
         \usepackage{wrapfig}
           \usepackage{graphicx}
             \usepackage{geometry}
               \usepackage[rightcaption]{sidecap}
                  \usepackage{tikz}
                   \usepackage{bbm}
						\usepackage{url}
\usepackage{amsfonts}
\usepackage{marvosym}
\usepackage{wasysym}
\usepackage{pifont}
\usepackage{mathrsfs}
\DeclareMathAlphabet{\mathpzc}{OT1}{pzc}{m}{it}

\allowdisplaybreaks[2]
\title{The Preservation of  Convexity by Geodesics in the Space of K\"ahler 
Potentials on Complex Affine Manifolds}
\author{Jingchen Hu}

\newtheorem{problem}{Problem}[section]

\newtheorem{proposition}{Proposition}[section]

\newtheorem{theorem}{Theorem}[section]

\newtheorem{remark}{Remark}[section]

\newtheorem{lemma}{Lemma}[section]
\numberwithin{equation}{section}
\newtheorem{definition}{Definition}[section]

\newtheorem{assumption}{Assumption}[section]
\newcommand{\ER}{\mathbb{R}}
\newcommand{\EZ}{\mathbb{Z}}
\newcommand{\EC}{\mathbb{C}}

\newcommand{\MR}{\mathcal{R}} 
\newcommand{\cMR}{{\overline{\MR}}}
\newcommand{\MV}{\mathcal{V}} 
\newcommand{\MH}{\mathcal{H}}
\newcommand{\MS}{\mathcal{S}}
\newcommand{\MG}{\mathcal{G}}
\newcommand{\ML}{\mathcal{L}}
\newcommand{\MD}{\mathcal{D}}
\newcommand{\OmegaPhi}{\Omega_0+\sqrt{-1}\ddbar\Phi}

\newcommand{\ddbar}{\partial\overline\partial}
\newcommand{\tautaubar}{\tau\overline{\tau}}

\newcommand{\taubar}{{\overline{\tau}}}

\newcommand{\abbar}{\alpha\overline{\beta}}

\newcommand{\ijbar}{i\overline j}

\newcommand{\gammabar}{{\overline{\gamma}}}
\newcommand{\alphabar}{{\overline{\alpha}}}
\newcommand{\thetabar}{{\overline{\theta}}}
\newcommand{\betabar}{{\overline{\beta}}}
\newcommand{\mubar}{{\overline{\mu}}}
\newcommand{\zetabar}{{\overline{\zeta}}}

\newcommand{\etabar}{{\overline{\eta}}}

\newcommand{\rhobar}{{\overline{\rho}}}

	\newcommand{\zbetabar}{\overline{z^{\beta}}}
\newcommand{\jbar}{{\overline{j}}}

\newcommand{\bea}{\begin{align}}
\newcommand{\ena}{\end{align}}

\newcommand{\tr}{\text{tr}}

\newcommand{\Abar}{{\overline{A}}}
	\newcommand{\Bbar}{{\overline{B}}}
	\newcommand{\MP}{\mathcal{P}}
	
	\newcommand{\MF}{\mathcal{F}}

	\newcommand{\SO}{$(S,\omega_0)$}
	\newcommand{\diag}{\text{diag}}
	\newcommand{\Abarinverse}{\overline{A^{-1}}}
	\newcommand{\Qsp}{Q^{<p>}_S}
	\newcommand{\Qp}{Q^{<p>}}
	\newcommand{\QsP}{Q^{[p]}_S}
	\newcommand{\Bi}{{\mathcal{B}_i}}
	\newcommand{\Bj}{{\mathcal{B}_j}}
	\newcommand{\Bjbar}{B_{\jbar}}

	\newcommand{\Bbari}{\overline{B}_i}
	\newcommand{\Ainverse}{{A^{-1}}}
	\newcommand{\Thetabar}{\overline{\Theta}}
	\newcommand{\iu}{\sqrt{-1}}
	\newcommand{\Ree}{\text{Re}}
	
	\newcommand{\Imm}{\text{Im}}
	\newcommand{\MB}{\mathcal{B}}
	\newcommand{\Bbariab}{\overline B_{i; \alpha\beta}}
	\newcommand{\mBiab}{\mathcal{ B}_{i; \alpha\beta}}
		\newcommand{\Bbarjab}{\overline B_{j; \alpha\beta}}
	\newcommand{\mBjab}{\mathcal{ B}_{j; \alpha\beta}}
\begin{document}
\maketitle
\begin{abstract}
On a compact complex affine manifold with a constant coefficient K\"ahler metric $\omega_0$, we introduce a concept: 
$(S,\omega_0)$-convexity and show that 
$(S,\omega_0)$-convexity is preserved by geodesics in the space of K\"ahler 
potentials. This implies that if two potentials are both strictly $(S,\omega_0)$-convex, then the metrics along the geodesic connecting them are non-degenerate.
\end{abstract}
\section{Introduction}\label{sec:Introduction}
Results of this paper provide partial answers to the following questions: First, in the space of K\"ahler potentials, with the metric introduced by Semmes-Mabuchi-Donaldson, any two points can be joined by a weak geodesic, but metrics along the geodesic may be degenerate. The question is  can we pose some conditions on two  points, so that metrics along the geodesic connecting them do not degenerate? Second, the maximum rank problem has been extensively studied for a general class of fully nonlinear elliptic equations. But the situation for degenerate elliptic equations has remained unexplored, for example, degenerate complex  Monge-Amp\`ere  equation. One question is whether, under some conditions, maximum rank property holds for solutions of degenerate complex Monge-Amp\`ere  equations? 

In this paper, on a complex affine manifold with a constant coefficient metric $\omega_0$, we introduce a concept \SO-convexity, and show that if two potentials are both strictly \SO-convex, then they can be connected by a geodesic with non-degenerate metric. Under similar condition, we can show the Hessian of the solution to the homogenous complex Monge-Amp\`ere equation on an $n+1$ dimensional product space has rank $n$.

 In section \ref{sec:background} we 
introduce some basic concepts and present some former results; in section 
\ref{sec:newconstructions} we introduce the concept of \SO-convexity, an elliptic perturbation of the homogenous complex Monge-Amp\`ere equation and also present our main results; in section \ref{sec:structure_of_the_paper} we provide a bird's-eye view of results of each section and show the structure of the paper; in section \ref{sec:more_notation}, more notation and convention are introduced.
\subsection{Background}\label{sec:background}
 Given an $n$ dimensional K\"ahler manifold $(V, \omega_0)$,
 we define the space of K\"ahler potentials:
 \begin{align}
 	\MH=\{\phi\in C^{\infty}(V)|\omega_0+\sqrt{-1}\ddbar\phi>0\}.
 \end{align}
A Riemannian metric can be introduced to this space. For $\psi_1, \psi_2\in T_\phi\MH$, let
\begin{align}
	<\psi_1, \psi_2>_{\phi}=\int_{V}\psi_1\psi_2(\omega_0+\sqrt{-1}\ddbar\phi)^{n}.
\end{align}
With the Riemannian metric above, for a curve $\{\varphi(t)|t\in[0,1]\}\subset \MH$,  
its 
energy is 
\begin{align}
	E(\varphi)=\int_{0}^1\int_\MV\varphi_t^2(\omega_0+\sqrt{-1}\ddbar\varphi)^n dt.
\end{align}
In this paper we denote  $\omega_0=\sqrt{-1}b_{\alpha\betabar}dz^{\alpha}\wedge 
\overline{dz^{\beta}}$, $\alpha,\beta\in\{1,...,n\}$, and 
$g_{\alpha\betabar}=b_{\alpha\betabar}+\varphi_{\alpha\betabar}, \ 
g^{\theta\betabar}g_{\alpha\betabar}=\delta_{\theta\alpha}$.
Then the Euler-Lagrange equation for the energy above is 
\begin{align}
	\varphi_{tt}=\varphi_{t\alpha}g^{\alpha\betabar}\varphi_{t\betabar}. \label{eq:geodesicEulerLagrange}
	\end{align}
 When $\omega_0+\sqrt{-1}\ddbar\varphi>0$, equation 
 (\ref{eq:geodesicEulerLagrange}) is equivalent to 
\begin{align}\label{eq:geodesicRealMonge-Ampere}
	\det\left(\begin{array}{cc}
		   \varphi_{tt}&  \varphi_{t\betabar}\\
	 \varphi_{ \alpha t}	  & \varphi_{\alpha\betabar}+b_{\alpha\betabar}
		             \end{array}
	              \right)=0.
\end{align}
Here the curve $\varphi$ is considered as a function defined on $[0,1]\times V$. Let
\begin{align}
	\MS=\{\tau=t+\sqrt{-1}\theta|0\leq  t\leq 1\}\subset \EC.
\end{align}
We can consider $\varphi$ as a function on $\MS\times V$
by letting $\varphi(\tau)=\varphi(t)$.  Then equation (\ref{eq:geodesicRealMonge-Ampere}) becomes a homogenous complex  Monge-Amp\`ere equation:
\begin{align}\label{eq:geodesicComplexMonge-Ampere}
	\det\left(\begin{array}{cc}
		\varphi_{\tautaubar}&\varphi_{\tau\betabar} \\
	 \varphi_{\alpha\taubar}	& \varphi_{\alpha\betabar}+b_{\alpha\betabar}
	\end{array}
	\right)=0.
\end{align}
Denote the projection from $\MS\times V$ to $V$ by $\pi_V$ and denote 
$\pi_V^\ast(\omega_0)$ by $\Omega_0$. Then equation 
(\ref{eq:geodesicComplexMonge-Ampere}) becomes
\begin{align}
	(\Omega_0+\sqrt{-1}\ddbar\varphi)^{n+1}=0.
\end{align}
This leads to the study of the following Dirichlet problem on $\MS\times V$.
\begin{problem}[Geodesic Problem]
	\label{prob:geodesicMongeAmpereonStrip}
	Given $\varphi_0,\varphi_1\in \MH$, find $\Phi\in C^{1,1}(\MS\times V)$, satisfying
	\begin{align}
			(\Omega_0+\sqrt{-1}\ddbar\Phi)^{n+1}=0, &\ \ \ \ \text{ in } \MS\times V;\\
\ \ \ 	\Omega_0+\sqrt{-1}\ddbar\Phi\geq 0, &\ \ \ \ \text{ in } \MS\times V;\\
\ \ \ \ \ \ \ 	\Phi_\theta=0, &\ \ \ \ \text{ in } \MS\times V;\\
\ \ \ \ \ \ \ 	\Phi=\varphi_0, &\ \ \ \ \text{ on } \{t=0\}\times V;\\
\ \ \ \ \ \ \ 	\Phi=\varphi_1, &\ \ \ \ \text{ on } \{t=1\}\times V.
\end{align}
\end{problem}
For a solution $\Phi$ to Problem \ref{prob:geodesicMongeAmpereonStrip},  $\Phi(t, 
\ast)$ may not be in $\MH$, so we consider it as a weak or generalized geodesic connection 
$\varphi_0$ and $\varphi_1$.
 
More generally, we can replace $\MS$ by a Riemann surface $\MR$ and consider the following Dirichlet problem. In this paper we only consider the case where $\MR$ is a bounded domain in $\EC$ with smooth boundary.
\begin{problem}[A Homogenous  Monge-Amp\`ere Equation on General Product 
Spaces]
	\label{prob:HCMAproductSpace}
	Given $\MR$, a bounded domain in $\EC$ with smooth boundary, and $F\in C^{\infty}(\partial \MR\times V)$ satisfying
	\begin{align}
		\omega_0+\sqrt{-1}\ddbar(F(\tau, \ast))>0, \ \ \ \ \text{ for any } 
		\tau\in\partial\MR,
	\end{align}find $\Phi\in C^{1,1}(\MR\times V)$ which satisfies
		\begin{align}
			\label{eq:HCMA_Product}
		(\Omega_0+\sqrt{-1}\ddbar\Phi)^{n+1}=0, &\ \ \ \ \text{ in } \MR\times V;\\
		\ \ \ 	\Omega_0+\sqrt{-1}\ddbar\Phi\geq0, &\ \ \ \ \text{ in } \MR\times V;\\
		\ \ \ \ \ \ \ 	\Phi=F, &\ \ \ \ \text{ on } \partial \MR\times V.
	\end{align}
\end{problem}
\begin{remark}\label{remark:geodesicProblemEquivalence}
	Problem \ref{prob:geodesicMongeAmpereonStrip} can be reduced to Problem 
	\ref{prob:HCMAproductSpace}. Let $f$ be a holomorphic covering map from 
	$\MS$ 
	to an annulus $\{\tau|\ 1< |\tau|<2\}.$ If $\Phi$ is a solution to Problem 
	\ref{prob:HCMAproductSpace} with $\MR$ being the annulus $\{\tau|\ 1< 
	|\tau|<2\}$ 
	and $F|_{\{|\tau|=1\}}=\varphi_0$, $F|_{\{|\tau|=2\}}=\varphi_1$, then 
	$\Phi(f(\tau), z)$ 
	is a solution to Problem \ref{prob:geodesicMongeAmpereonStrip}.
\end{remark}

Problem \ref{prob:HCMAproductSpace} and \ref{prob:geodesicMongeAmpereonStrip}
 were introduced by \cite{Mabuchi} \cite{Semmes} \cite{DonaldsonSymmetricSpace}. The existence of $C^{1,1}$ solution was established by \cite{Chen2000} \cite{Blocki} \cite{JianchunFirst} \cite{Jianchun} and it was also shown by \cite{LempertLizVivas} and \cite{LempertDarvas} that the optimal regularity of general solutions is $C^{1,1}$.  
 
 Besides regularity, we may ask if 
 \begin{align}
 	\omega_0+\iu\ddbar\Phi(\tau, \ast)>0, \ \ \ \ \ \ \text{for all }\tau\in\MR,
 	\label{1119_118}
 \end{align}
for a solution $\Phi$ to Problem \ref{prob:HCMAproductSpace}. This is similar to the maximum rank problem, which asks if
\begin{align}
  \text{rank}(	\Omega_0+\iu\ddbar\Phi(\tau, \ast))=n.
   	\label{1119_119}
\end{align} 
It's easy to see that (\ref{1119_118}) implies (\ref{1119_119}). But the inverse implication may not be true.  It turns out that, for a general solution, (\ref{1119_118}) or (\ref{1119_119}) may not be valid. An example was constructed in \cite{RossNystrom}, when $\MR$ is a disc and $V$ is $\EC P^1$. In this example, a solution $\Phi$ was constructed, which satisfies
\begin{align}
	\Omega_0+\iu\ddbar\Phi=0,
\end{align}
in an open set in $\MR\times V$. However, we may ask is it possible to find some conditions for the boundary value $F$, so that if they are satisfied then (\ref{1119_118}) is valid.

Theorem 1 of \cite{DonaldsonHolomorphicDiscs} says that, when $\MR$ is a disc,  the set of smooth functions $F$ for which a smooth solution to Problem \ref{prob:HCMAproductSpace} exists is open in $C^{\infty}(\partial \MR\times V)$.  Actually the proof implies  that if the boundary value of $\Phi$ is in this set, then
(\ref{1119_118}) is satisfied. The proof also suggests that this set is open in $C^2$ topology.  The proof made use of the foliation structure associated to a solution to homogenous complex Monge-Amp\`ere equations. This technique was also used in \cite{Lempertmetrique}, to construct pluri-complex Green's function. In \cite{CFH}, by partially generalizing this technique to the case where $\MR$ is an annulus , we proved that if $|\varphi_0|_{5}+|\varphi_1|_{5}$ is small enough, then the geodesic connecting  $\varphi_0$ and $\varphi_1$ are $C^4$ and 
  \begin{align}
 	\omega_0+\iu\ddbar\Phi(t, \ast)>0, \ \ \ \ \ \ \text{for all }t\in[0,1].
 	\label{1119_121}
 \end{align}
In a recent paper \cite{HuC2Perturb}, the author improved the result above, reducing the $C^5$ smallness condition to a $C^2$ smallness condition.

However, \cite{GuanPhongMaxRank} \cite{LempertLizVivas}  \cite{HuIMRN}
and the Appendix A of \cite{HuC2Perturb} all suggest that a proper condition on $F$, which implies (\ref{1119_118}), may be a convexity condition. In \cite{GuanPhongMaxRank}, it was shown that,
when $V$ is a 1-dimensional complex flat torus, if $\varphi_1$ and $\varphi_0$ both satisfy a convexity condition, then the geodesic connecting them has non-degenerate metric. In the Appendix A of \cite{HuC2Perturb}, a  similar result was proved, for solutions to Problem \ref{prob:HCMAproductSpace}, with a very different method. 
Furthermore,  computations of \cite{LempertLizVivas} and \cite{HuIMRN} suggest that the convexity condition is also necessary.

In this paper, when $V$ is a compact complex affine manifold with a constant coefficient metric, we introduce a concept: \SO-convexity and show that boundary values satisfying \SO-convexity implies (\ref{1119_118}).
\subsection{Notation, Constructions and Main Results}\label{sec:newconstructions}
In this paper, we will discuss the situation, where $V$ is a compact complex affine manifold with a constant coefficient K\"ahler metric $\omega_0$. By the definition of complex affine manifold, $V$ is equipped with an atlas, so that all transition maps are affine and holomorphic.  
Furthermore, we require that $\omega_0$ is a constant metric, which we mean, in any coordinate neighborhood, with coordinates $\{z^\alpha\}_{\alpha=1}^n$, if
\begin{align}
	\omega_0=\sqrt{-1}b_{\alpha\betabar}dz^\alpha \wedge \overline{dz^\beta},
\end{align}
then $b_{\alpha\betabar} $ is constant for any $\alpha, \beta\in \{1,...,n\}.$

With these preparations, we can introduce the concept of $\omega_0$-convexity.
\begin{definition}
	[$\omega_0$-Convexity and Strict $\omega_0$-Convexity]
	\label{def:OmegaZeroConvexityC0_1117}
	 A function $\varphi\in C^{0}(V)$ is (strictly) $\omega_0$-convex if, in any coordinate chart, with coordinates $\{z^\alpha\}$,
	\begin{align}
	\varphi+b_{\alpha\betabar}z^\alpha\overline{z^\beta}
	\end{align}
is a (strictly) convex function.
 \end{definition}
The convexity above has been widely used in many works related to Hessian manifolds, for example, it was called local convexity in \cite{CaffarelliHessian}  and called g-convexity in \cite{GuedjToConvexity}. However, to estimate the convexity of solutions to Problem \ref{prob:HCMAproductSpace}, it's necessary to extend this concept and introduce the following concept of \SO-convexity.
\begin{definition}
	[\SO-Convexity and Strict \SO-Convexity for $C^0$ Function]
	\label{def:SOmegaZeroConvexityC0_1117}
	Suppose $S$ is a constant section of $T_{2,0}^{\ast}(V)$. Then a function $\varphi\in C^{0}(V)$ is (strictly) \SO-convex if, in any coordinate chart, with coordinates $\{z^\alpha\}$,
	\begin{align}
		\varphi+b_{\alpha\betabar}z^\alpha\overline{z^\beta}+\Ree(S_{{\alpha\beta}}z^\alpha z^\beta)
	\end{align}
	is a (strictly) convex function.
\end{definition} For a constant section, we mean, in any coordinate chart, the tensor components of $S$ are constant. 
Obviously, when $S=0$, \SO-convexity is exactly $\omega_0$-convexity.
Furthermore, to gauge the convexity, we  introduce the concept of modulus of convexity.
\begin{definition}
	[Modulus of \SO-Convexity]
	\label{def:Module_SOmegaZeroConvexityC0_1117}
	Suppose $S$ is a constant section of $T_{2,0}^{\ast}(V)$. Then a function $\varphi\in C^{0}(V)$ is  \SO-convex of modulus $\geq\mu$ if, in any coordinate chart, with coordinates $\{z^\alpha\}$,
	\begin{align}
		\varphi+(1-\mu)b_{\alpha\betabar}z^\alpha\overline{z^\beta}+\Ree(S_{{\alpha\beta}}z^\alpha z^\beta)  \label{125_1120}
	\end{align}
is a  convex function. It's \SO-convex of modulus $>\mu$ if (\ref{125_1120}) is strictly convex.
\end{definition}
\begin{remark}
	It's easy to see that if $\varphi$ is \SO-convex of modulus $>\mu$, for $\mu\geq 0$, then it's strictly \SO-convex. Another fact is that $\varphi$ is \SO-convex of modulus $>\mu$ implies that it's \SO-convex of  modulus $\geq \mu$. In addition, since $V$ is compact, $\varphi$ is \SO-convex of  modulus $\geq \mu$ implies that, for any $\mu'<\mu$, $\varphi$ is \SO-convex of modulus $> \mu'$.
\end{remark}

When the function $\varphi$ is $C^2$ continuous, the strict \SO-convexity can be defined using complex second order derivatives. Using Lemma \ref{lemma:equivalent_definition_S_Omega_Convexity}, we will see
Definition \ref{def:SOmegaZeroConvexityC0_1117} is equivalent to the following.
\begin{definition}
	[Strict \SO-Convexity for $C^2$ Functions]
	\label{def:SOmegaConvexityMatrix_1117} 	Suppose $S$ is a constant section of $T_{2,0}^{\ast}(V)$. Then a function $\varphi\in C^{2}(V)$ is   strictly \SO-convex  if
	\begin{align}
	\omega_0+\sqrt{-1}\ddbar\varphi>0
	\end{align}
and the maximal eigenvalue of the tensor \begin{align}
	K=(\varphi_{\theta\nu}-S_{\theta\nu}) g^{\nu\zetabar}\overline{(\varphi_{\eta\zeta}-S_{\eta\zeta})}g^{\gamma\etabar}dz^\theta\otimes \frac{ \partial}{\partial z^{\gamma}}
\end{align} is smaller than 1.
\end{definition}

\begin{remark}
	Here $\varphi_{\alpha\beta}dz^\alpha\otimes dz^\beta$ is a well defined section of $T_{2,0}^\ast(V)$, because the coordinate transition functions between charts are affine.  In addition, using basic linear algebra, we know eigenvalues of $K$ are real and non-negative.
\end{remark}
We introduce another measurement of the convexity:
\begin{definition}
[Degree of Convexity]
\label{def:Degree_SO_convexity}
Suppose $S$ is a constant section of $T^\ast_{2,0}(V)$ and $\delta$ is a positive number. Then a function $\varphi\in 
C^{2}(V)$ is \SO-convex of degree $>\delta$ if, for any constant section $\Theta$ of $T_{2,0}^\ast(V)$ with
\begin{align}
\text{maximum eigenvalue of }	\left(\Theta_{\alpha\eta} b^{\eta\gammabar}\  \overline{\Theta_{\rho\gamma}} b^{\beta\rhobar} dz^\alpha\otimes \frac{\partial }{\partial z^{\beta}}\right)\leq \delta^2,
\end{align}  $\varphi$ 
is a strictly $(S+\Theta,\omega_0)$-convex function.
\end{definition}
It turns out that the degree of convexity and the modulus of convexity coincide. In Lemma \ref{lemma:Equivalence Between Module of Convexity and Degree of Convexity}, we show that,  for $\varphi\in C^2(V)$ and $\delta\geq 0$,  $\varphi$ is \SO-convex of degree $>\delta$ if and only if it is \SO-convex of modulus $>\delta$.

The main results of the paper are the following.
\begin{theorem}
	[Estimates for Geodesics]
	\label{thm:estimates_for _Geodesics}
	Given $\varphi_0,\varphi_1\in \MH$, suppose that there is an $S$ which is  a constant section of $T^\ast_{2,0}(V)$, so that $\varphi_0$ and $\varphi_1$ are both  \SO-convex of modulus $>\mu$, for $\mu> 0$. Let $\{\varphi_t|\ t\in[0,1]\}$ be the geodesic connecting $\varphi_0$ and $\varphi_1$. Then, for any $t\in (0,1)$,
	$\varphi_t$ is \SO-convex of modulus $\geq\mu$ and, by definition, this implies
	\begin{align}
		\omega_0+\sqrt{-1}\ddbar\varphi_t\geq \mu\omega_0,
	\end{align}
in the weak sense.
\end{theorem}
The theorem above is a particular case of the following theorem, according to Remark \ref{remark:geodesicProblemEquivalence}.
\begin{theorem}
	[Estimates on Product Space]
	\label{thm:estimate_HCMA_ProductSpace}
	Suppose $F$ is a $C^\infty$ function on $\partial\MR\times V$ and, for a constant $\mu > 0$ and a constant section $S$ of $T^\ast_{2,0}(V)$, 
	\begin{align} \label{condition:convex_module_theorem}
		\text{$F(\tau,\ast)$ is \SO-convex of modulus $>\mu$, for any $\tau\in\partial\MR$.}
	\end{align}
	  Let $\Phi$ be the solution to Problem \ref{prob:HCMAproductSpace} with boundary value $F$. Then
    $\Phi(\tau, \ast)$ is \SO-convex of modulus $\geq\mu$, for any $\tau\in\MR$, and, as a consequence
    \begin{align}
    	\omega_0+\sqrt{-1}\ddbar[\Phi(\tau,\ast)]\geq \mu\omega_0,
    \end{align}
	in the weak sense.
\end{theorem}

As we discussed, solutions to Problem \ref{prob:HCMAproductSpace} may only be $C^{1,1}$, but to implement our method we need to use up to 4-th order derivatives. Therefore, we consider an elliptic perturbation of Problem \ref{prob:HCMAproductSpace}, for which there 
are 
smooth solutions. In many previous works, the following problem was studied:
\begin{problem}[Non-Degenerate Monge-Amp\`ere Equation on Product Space]
	\label{prob:epsilonCMAproductSpace}
	Given $\MR$, a bounded domain in $\EC$ with smooth boundary, and $F\in C^{\infty}(\partial \MR\times V)$ satisfying
	\begin{align}
		\omega_0+\sqrt{-1}\ddbar(F(\tau, \ast))>0, &\ \ \ \ \text{ for any } \tau\in\partial\MR,
	\end{align}find $\Phi\in C^{1,1}(\MR\times V)$ which satisfies
	\begin{align}
		(\Omega_0+\sqrt{-1}\ddbar\Phi)^{n+1}&=\varepsilon\sqrt{-1}d\tau\wedge 
		\overline{d\tau}\wedge \Omega_0^n, &\ \ \ \ \text{ in } \MR\times V;\\
		\ \ \ 	\Omega_0+\sqrt{-1}\ddbar\Phi&>0, &\ \ \ \ \text{ in } \MR\times V;\\
		\ \ \ \ \ \ \ 	\Phi&=F, &\ \ \ \ \text{ on } \partial \MR\times V.
	\end{align}
In above, $\varepsilon$ is a positive constant. 
\end{problem}
 However, in this paper we introduce a 
different perturbation.
\begin{problem}
	[An Elliptic Perturbation of Homogenous Complex Monge-Amp\`ere Equation]
	\label{prob:NewPerturbation}
	Suppose $F\in C^{\infty}(\partial \MR\times V)$ satisfies that 
		\begin{align}
		\omega_0+\sqrt{-1}\ddbar(F(\tau, \ast))>0, \ \ \ \ \text{ for any } \tau\in\partial\MR.
	\end{align} Find  $\Phi\in C^{\infty}(\MR\times V)$ satisfying:
\begin{align}
 	&(\Omega_0+\sqrt{-1}\ddbar\Phi)^{n+1}=\epsilon \sqrt{-1} d\tau\wedge\overline{d\tau}\wedge\Omega_0\wedge(\Omega_0+\sqrt{-1}\ddbar\Phi)^{n-1}, & \text{ in\ \ \  } \MR\times V;
 	\label{eq:new_Perturbation_in_Form_Form_Section1}\\
 	&\ \ \ \ \ \  \ \ \ \ 	\Omega_0+\sqrt{-1}\ddbar\Phi>0, & \text{ in\ \ \ } \MR\times V;\label{eq:Problem_NewPerturbation_NonDegenerateCondition}\\
 	&\ \ \ \ \ \ \ \ \ \  \ \ \ \ 	\Phi=F, &\text{ on \ } \partial \MR\times V.
\end{align}
Here $\epsilon$ is a positive constant.
\end{problem}

\begin{remark}
		Equation (\ref{eq:new_Perturbation_in_Form_Form_Section1}) is equivalent to 
	\begin{align}
		\Phi_{\tau\taubar}-\Phi_{\tau\betabar}g^{\alpha\betabar}\Phi_{\alpha\taubar}=\epsilon b_{\alpha\betabar}g^{\alpha\betabar},
		\label{eq:new_Perturbation_in_Inverse_Form_Section1}
	\end{align}
providing
	\begin{align}
	\omega_0+\sqrt{-1}\ddbar(\Phi(\tau, \ast))>0, \ \ \ \ \text{ for any } \tau\in\MR.
\end{align} We also notice that, when $n=1$, Problem 
\ref{prob:NewPerturbation} is exactly Problem \ref{prob:epsilonCMAproductSpace}, 
with $\varepsilon=\epsilon$. In a previous paper \cite{HuC2Perturb} we proved a 
particular 1-dimensional case of Theorem \ref{thm:estimate_HCMA_ProductSpace}, 
by 
working with solutions to Problem \ref{prob:epsilonCMAproductSpace}.
\end{remark}
In Section \ref{sec:Existence}, we prove that under the condition of Theorem \ref{thm:estimate_HCMA_ProductSpace} a smooth solution to Problem \ref{prob:NewPerturbation} exists. However, we don't know if a smooth solution always exists for general boundary value $F$. The convergence of solutions as  $\epsilon$ goes to zero is discussed in Section \ref{sec:estimates_No_Assumption}.

\subsection{Structure of the Paper}\label{sec:structure_of_the_paper}
In section \ref{sec:Equation},  by differentiating equation (\ref{eq:new_Perturbation_in_Inverse_Form_Section1}), we find second order derivatives of $\Phi$, 
\begin{align}
	A_{{\alpha\betabar}}=\Phi_{\alpha\betabar} +b_{\alpha\betabar}\ \ \ \ \ \  \text{and} \ \ \ \ \ \ B_{{\alpha\beta}}=\Phi_{\alpha\beta},
\end{align}
satisfy two non-linear equations (\ref{eq:A_section2_first_Matrix}) and (\ref{eq:B_section2_first_Matrix}).

In section \ref{sec:computation_for_Apriori_Estimate}, using $A$ and $B$, we construct $M_S$, which measures the convexity, and $\QsP$, which is a smooth approximation of $M_S$. Then, using equation (\ref{eq:A_section2_first_Matrix}) and (\ref{eq:B_section2_first_Matrix}), we show, for an elliptic operator $L^{\ijbar}\partial_{\ijbar}$ and $\Qsp=\left(\QsP\right)^p$,
\begin{align}
	L^{\ijbar}\partial_{\ijbar}\left(\Qsp\right)\geq 0,
\end{align}
providing $\QsP\leq 1-\frac{1}{2p}$.  

In section \ref{sec:method_continuity}, using a continuity argument, we show if $M_S<1$ on ${\partial \MR\times V}$, then $M_S<1$ in ${ \MR\times V}$. This means the strict \SO-convexity on $\partial\MR\times V$ implies the strict \SO-convexity in $\MR\times V$.

In section \ref{sec:metriclowerBound}, by altering $S$, we show the  \SO-convexity of modulus $>\mu$ on $\partial\MR\times V$ implies the  \SO-convexity of modulus $>\mu$ in $\MR\times V$. 

We point out that estimates in section \ref{sec:computation_for_Apriori_Estimate}, \ref{sec:method_continuity} and \ref{sec:metriclowerBound} all depends on some apriori assumptions. These assumptions can be removed after we prove the existence of smooth solutions to Problem \ref{prob:NewPerturbation}. 

In section \ref{sec:Existence}, we establish $C^0$, $C^1$, $C^2$ and $C^{2,\alpha}$ estimates. The methods are standard in PDE. In the proof, the result of section \ref{sec:metriclowerBound} plays an import role. Actually, it basically says equation (\ref{eq:new_Perturbation_in_Inverse_Form_Section1}) is a uniform elliptic equation, providing $\epsilon>0$.
With these estimates, we prove existence of $C^\infty$ solution by the method of continuity in section \ref{sec:existence_by_method_continuity}.

Finally, we prove estimates for solutions to Problem \ref{prob:geodesicMongeAmpereonStrip} and \ref{prob:HCMAproductSpace}, by letting $\epsilon\rightarrow 0.$

 \subsection{A Convention for Tensor Contraction}\label{sec:more_notation}
In this paper, we need to contract a sequence of rank $2$ tensors to form new tensors. The computation can be simplified by converting rank $2$ tensors to matrices. In this sections, we explain how to do this.

Suppose $A$ is a section of $T_{1,1}^{\ast}(V)$ and $B$ is a section of $T_{2,0}^{\ast}(V)$. In a coordinate chart, $A$ and $B$ can be considered as matrix valued functions, which we still denote by $A $ and $B$. Let
\begin{align}
	A=(A_{{\alpha\betabar}}),\ \ \ \ \ \ B=(B_{\alpha\beta}),
\end{align}
where $\alpha$ is the row index and $\beta$ is the column index. As a matrix, when $A$ is invertible, we denote
\begin{align}
	(A^{{\alpha\betabar}})=A^{-1},
\end{align}
where $\alpha$ is the column index and $\beta$ is the row index.
Then we have
\begin{align}
	A^{\alpha\betabar} A_{\alpha\thetabar}=\delta_{\theta\beta}
\end{align}
and we know 
\begin{align}
	A^{\alpha\betabar}\frac{\partial}{\partial z^\alpha}\otimes \overline{\frac{\partial}{\partial z^\beta}}
						\label{141_1119}
\end{align}
is a section of $T_{1,1}(V)$. 

Let 
\begin{align}
	K_\alpha^\beta=B_{{\alpha\theta}} \overline{A^{\mu\thetabar}}\  \overline{B_{\mu\rhobar}} A^{\beta\rhobar},
\end{align}then $K_\alpha^\beta dz^\alpha\otimes \frac{\partial}{\partial z^\beta}$ is a section of $T_{1,0}^\ast(V)\otimes T_{1,0}(V)$.  Locally, we can consider $K$ as a matrix, with
\begin{align}
	K=(K_{\alpha}^{\beta}),
\end{align}
where $\alpha$ is the row index and $\beta$ is the column index.
Then
\begin{align}
	K=B\Abarinverse\  \Bbar \Ainverse.
\end{align}
It's easy to see that eigenvalues of $K$ do not change under coordinate transformations. Furthermore, for $p\in \EZ^+$, we need to use $K^p$, which is also a section of $T_{1,0}^\ast(V)\otimes T_{1,0}(V)$.

In addition, for a section $\Theta$ of $T_{1,0}^\ast(V)\otimes T_{1,0}(V)$, we denote 
\begin{align}
	\Theta<C, 
\end{align}
for a constant $C$, if,  at any point $p\in V$, the maximum eigenvalue of $\Theta(p)$ is smaller than $C$.


\section{Equations for Second Order Derivatives}
            \label{sec:Equation}
            In this paper we will mainly work on the product space $\MR\times V$, where $\MR$ is a compact domain in $\EC$ with smooth boundary and $V$ is the complex affine manifold. We denote the coordinate on $\MR$ by $\tau$ and denote the coordinates on $V$ by $\{z^\alpha\}_{\alpha=1}^n$. Coordinates on $V$ are indexed by Greek letters, except $\tau$. The coordinate $\tau$ on $\MR$ will be considered as the $0$-th coordinate and, in some situation, we denote it by $z^0$. Thus, the coordinates on $\MR\times V$ will be indexed by Roman letters, running from $0$ to $n$.

            Suppose $\Phi$ is a solution to Problem \ref{prob:NewPerturbation}. Let 
            \begin{align}
            &	A_{\alpha\betabar}=\Phi_{\alpha\betabar}+b_{\alpha\betabar}, \ \ \ \ \ \ \ \ 
            	B_{\alpha\beta}=\Phi_{\alpha\beta}.&
            \end{align}
        As described in section \ref{sec:more_notation}, $A, B$ can be considered as matrices locally.
        In this section, by differentiating (\ref{eq:new_Perturbation_in_Inverse_Form_Section1}), we show, as matrices, $A, B$ satisfy the following equations:
        \begin{align}
        	L^{i\jbar}\partial_{i\jbar} A=L^{i\jbar} (\partial_i A) A^{-1} (\partial_\jbar A)+ L^{\ijbar} (\partial_\jbar B) \overline{A^{-1}} (\partial_i \overline B),
        	\label{eq:A_section2_first_Matrix}
        	\\
        	L^{i\jbar}\partial_{i\jbar} B=L^{i\jbar} (\partial_i A) A^{-1} (\partial_\jbar B)+ L^{\ijbar} (\partial_\jbar B) \overline{A^{-1}} (\partial_i \overline A).
        	\label{eq:B_section2_first_Matrix}
        \end{align}
   Here $L=L^{i\jbar}\partial_{\ijbar}$ is an elliptic operator on $\MR\times V$, with
   \begin{align}  			\label{eq:L_introduced}
   	\left(
   	\begin{array}{cc}
   		L^{0\overline 0} & L^{0\betabar}\\
   		L^{\alpha\overline 0}& L^{\alpha\betabar}
   	\end{array}
   	\right)=   	\left(
   	\begin{array}{cc}
   		1 & -\Phi_{\mu\taubar}g^{\mu\betabar}\\
   		 -\Phi_{\tau\mubar}g^{\alpha\mubar}&\Phi_{\tau\mubar}g^{\alpha\mubar} \Phi_{\mu\taubar}g^{\mu\betabar}+\epsilon b_{\eta\zetabar} g^{\alpha\zetabar}g^{\eta\betabar}
   	\end{array}
   	\right).
   \end{align}

The computations in this section are similar to the computations of Section 2.1 and 
2.2 
of \cite{HuC2Perturb}. However, the computations here are much simpler because 
we 
are working on a flat affine manifold so $\Phi_{\alpha\beta}$ are coordinate derivatives while 
in 
\cite{HuC2Perturb} we need to compute covariant derivatives. 

Apply $\partial_\theta$ to (\ref{eq:new_Perturbation_in_Inverse_Form_Section1}), we get 
\begin{align}
	\Phi_{\theta\tau\taubar}-\Phi_{\theta\betabar\tau}g^{\alpha\betabar}\Phi_{\alpha\taubar}
	+\Phi_{\tau\betabar}g^{\alpha\mubar}\Phi_{\rho\mubar\theta}g^{\rho\betabar}\Phi_{\alpha\taubar}-\Phi_{\tau\betabar}
	 g^{\alpha\betabar}\Phi_{\alpha\taubar\theta}=-\epsilon b_{\alpha\betabar 
	}g^{\alpha\mubar}\Phi_{\rho\mubar\theta}g^{\rho\betabar}.
	\label{eq:D_theta_NewPerturbation}
\end{align}
Then apply $\partial_\gamma$ to (\ref{eq:D_theta_NewPerturbation}), we get
\begin{align}
	\label{eq:20221030-1}
	&\Phi_{\theta\gamma\tautaubar}-\Phi_{\theta\gamma\betabar\tau}g^{\alpha\betabar}\Phi_{\alpha\taubar}+\Phi_{\theta\betabar\tau}g^{\alpha\mubar}\Phi_{\gamma\mubar\rho}g^{\rho\betabar}\Phi_{\alpha\taubar}
	-\Phi_{\theta\betabar\tau}g^{\alpha\betabar}\Phi_{\alpha\gamma\taubar}+\Phi_{\tau\gamma\betabar}g^{\alpha\mubar}\Phi_{\rho\mubar\theta}g^{\rho\betabar}\Phi_{\alpha\taubar}
	\\
	\label{eq:20221030-2}
	&-\Phi_{\tau\betabar} g^{\alpha\etabar} \Phi_{\zeta\etabar\gamma}g^{\zeta\mubar}\Phi_{\rho\mubar\theta} g^{\rho\betabar} \Phi_{\alpha\taubar}+\Phi_{\tau\betabar} g^{\alpha\mubar}\Phi_{\theta\gamma\rho\mubar} g^{\rho\betabar} \Phi_{\alpha\taubar}-\Phi_{\tau\betabar}g^{\alpha\mubar} \Phi_{\rho\mubar\theta} g^{\rho\etabar}\Phi_{\zeta\etabar\gamma}g^{\zeta\betabar}\Phi_{\alpha\taubar}
	\\
	 \label{eq:20221030-3}
	&+\Phi_{\tau\betabar}g^{\alpha\mubar} \Phi_{\rho\mubar\theta} g^{\rho\betabar} \Phi_{\alpha\gamma\taubar}-\Phi_{\tau\gamma\betabar} g^{\alpha\betabar}\Phi_{\alpha\theta\taubar}
	+\Phi_{\tau\betabar}g^{\alpha\mubar}\Phi_{\zeta\mubar\gamma}g^{\zeta\betabar}\Phi_{\theta\alpha\taubar}-\Phi_{\tau\betabar}
	 g^{\alpha\betabar}\Phi_{\alpha\theta\gamma\taubar}
	 	\\
	 	\label{eq:20221030-4}
	&\ \ =\epsilon b_{\alpha\betabar }g^{\alpha\zetabar}\Phi_{\eta\zetabar\gamma}g^{\eta\mubar}\Phi_{\rho\mubar\theta}g^{\rho\betabar}-\epsilon b_{\alpha\betabar} g^{\alpha\mubar} \Phi_{\theta\gamma\rho\mubar}g^{\rho\betabar}
	+\epsilon b_{\alpha\betabar} g^{\alpha\mubar}\Phi_{\rho\mubar\theta}g^{\rho\zetabar}\Phi_{\eta\zetabar\gamma}g^{\eta\betabar}.
\end{align}

We will use the following convention. $(\ast.\ast)_k$ stands for the $k$-th term in the line $(\ast.\ast)$, including the sign. For example, 
\begin{align}
	(\ref{eq:20221030-4})_2=-\epsilon b_{\alpha\betabar} g^{\alpha\mubar} \Phi_{\theta\gamma\rho\mubar}g^{\rho\betabar}, \ \ \ \ 
		(\ref{eq:20221030-1})_1=\Phi_{\theta\gamma\tau\taubar}, \ \ \ \ 	(\ref{eq:20221030-3})_4=-\Phi_{\tau\betabar}
		g^{\alpha\betabar}\Phi_{\alpha\theta\gamma\taubar}.
\end{align}
It's straightforward to verify the following equality:
\begin{align}
	(\ref{eq:20221030-1})_1
	+(\ref{eq:20221030-1})_2
	+(\ref{eq:20221030-2})_2
	+(\ref{eq:20221030-3})_4
	-(\ref{eq:20221030-4})_2
	&=
	\Phi_{\theta\gamma i \jbar}L^{i\jbar},\\
	(\ref{eq:20221030-1})_3
	+(\ref{eq:20221030-1})_4
	+(\ref{eq:20221030-2})_1
	+(\ref{eq:20221030-3})_1
	-(\ref{eq:20221030-4})_1
	&=-\Phi_{\theta\betabar i }g^{\alpha\betabar}\Phi_{\alpha\gamma\jbar}L^{i\jbar},\\
	(\ref{eq:20221030-1})_5
	+(\ref{eq:20221030-2})_3
	+(\ref{eq:20221030-3})_2
	+(\ref{eq:20221030-3})_3
	-(\ref{eq:20221030-4})_3
	&=
	-\Phi_{\theta\rho\jbar}g^{\rho\zetabar}\Phi_{\gamma\zetabar i}L^{i\jbar}.
\end{align}
This gives us (\ref{eq:B_section2_first_Matrix}).

Similarly, we apply $\partial_\gammabar$ to (\ref{eq:D_theta_NewPerturbation}) and 
get
\begin{align}
	\label{eq:20221032-1}
	&\Phi_{\theta\gammabar\tautaubar}-\Phi_{\theta\gammabar\betabar\tau}g^{\alpha\betabar}\Phi_{\alpha\taubar}+\Phi_{\theta\betabar\tau}g^{\alpha\mubar}\Phi_{\gammabar\mubar\rho}g^{\rho\betabar}\Phi_{\alpha\taubar}
	-\Phi_{\theta\betabar\tau}g^{\alpha\betabar}\Phi_{\alpha\gammabar\taubar}+\Phi_{\tau\gammabar\betabar}g^{\alpha\mubar}\Phi_{\rho\mubar\theta}g^{\rho\betabar}\Phi_{\alpha\taubar}
	\\
	\label{eq:20221032-2}
	&-\Phi_{\tau\betabar} g^{\alpha\etabar} 
	\Phi_{\zeta\etabar\gammabar}g^{\zeta\mubar}\Phi_{\rho\mubar\theta} 
	g^{\rho\betabar} \Phi_{\alpha\taubar}+\Phi_{\tau\betabar} 
	g^{\alpha\mubar}\Phi_{\theta\gammabar\rho\mubar} g^{\rho\betabar} 
	\Phi_{\alpha\taubar}-\Phi_{\tau\betabar}g^{\alpha\mubar} 
	\Phi_{\rho\mubar\theta} 
	g^{\rho\etabar}\Phi_{\zeta\etabar\gammabar}g^{\zeta\betabar}\Phi_{\alpha\taubar}
	\\
	\label{eq:20221032-3}
	&+\Phi_{\tau\betabar}g^{\alpha\mubar} \Phi_{\rho\mubar\theta} 
	g^{\rho\betabar} 
	\Phi_{\alpha\gammabar\taubar}-\Phi_{\tau\gammabar\betabar} 
	g^{\alpha\betabar}\Phi_{\alpha\theta\taubar}
	+\Phi_{\tau\betabar}g^{\alpha\mubar}\Phi_{\zeta\mubar\gammabar}g^{\zeta\betabar}\Phi_{\theta\alpha\taubar}-\Phi_{\tau\betabar}
	g^{\alpha\betabar}\Phi_{\alpha\theta\gammabar\taubar}
	\\
	\label{eq:20221032-4}
	&\ \ =\epsilon b_{\alpha\betabar 
	}g^{\alpha\zetabar}\Phi_{\eta\zetabar\gammabar}g^{\eta\mubar}\Phi_{\rho\mubar\theta}g^{\rho\betabar}-\epsilon
	 b_{\alpha\betabar} g^{\alpha\mubar} 
	\Phi_{\theta\gammabar\rho\mubar}g^{\rho\betabar}
	+\epsilon b_{\alpha\betabar} 
	g^{\alpha\mubar}\Phi_{\rho\mubar\theta}g^{\rho\zetabar}\Phi_{\eta\zetabar\gammabar}g^{\eta\betabar}.
\end{align}
It's straightforward to verify the following equality
\begin{align}
	(\ref{eq:20221032-1})_1
	+(\ref{eq:20221032-1})_2
	+(\ref{eq:20221032-2})_2
	+(\ref{eq:20221032-3})_4
	-(\ref{eq:20221032-4})_2
	&=
	\Phi_{\theta\gammabar i \jbar}L^{i\jbar},\\
	(\ref{eq:20221032-1})_3
	+(\ref{eq:20221032-1})_4
	+(\ref{eq:20221032-2})_1
	+(\ref{eq:20221032-3})_1
	-(\ref{eq:20221032-4})_1
	&=-\Phi_{\theta\betabar i 
	}g^{\alpha\betabar}\Phi_{\alpha\gammabar\jbar}L^{i\jbar},\\
	(\ref{eq:20221032-1})_5
	+(\ref{eq:20221032-2})_3
	+(\ref{eq:20221032-3})_2
	+(\ref{eq:20221032-3})_3
	-(\ref{eq:20221032-4})_3
	&=
	-\Phi_{\theta\rho\jbar}g^{\rho\zetabar}\Phi_{\gammabar\zetabar i}L^{i\jbar}.
\end{align}
This gives us (\ref{eq:A_section2_first_Matrix}).
\section{Apriori Estimates}
            \label{sec:AprioriEstimate}
           In this section, we will prove some estimates for solutions to Problem 
           \ref{prob:NewPerturbation}, with some apriori assumptions. These assumptions can be removed after we prove the existence 
           of $C^\infty$ solutions to  Problem \ref{prob:NewPerturbation} in section 
           \ref{sec:Existence}. One estimate in this section, Prop \ref{prop:metric_lower_bound}, will play an indispensable role in our proof of the existence result.
           
           Suppose $S$ is a constant section of $T_{2,0}^\ast(V)$ and $F\in C^{\infty}(\partial \MR\times V)$ satisfies
           \begin{align}
           	F(\tau,\ast) \text{ is strictly \SO-convex, for any 
           		$\tau\in \partial \MR$. }\label{20221107BoundarySOmega_0Convexity}
           \end{align}
       For a solution $\Phi$ to Problem \ref{prob:NewPerturbation} with boundary value 
       $F$, let 
       \begin{align}
       		B&=(\Phi_{\alpha\beta}),\\
       	A&=(\Phi_{\alpha\betabar}+b_{\alpha\betabar}),\\
       	K_S&=(B-S)\Abarinverse\ \overline{(B-S)} \Ainverse,\label{defining_KS}
       \end{align}
   and, for any $(\tau, z)\in \MR\times V$,
   \begin{align}
   	M_S(\tau, z)=\text{ Maximum Eigenvalue of }K_S(\tau, z).
   \end{align}   Condition (\ref{20221107BoundarySOmega_0Convexity}) implies that 
       \begin{align}
       	M_S<1,    \text{ on } \partial \MR\times V.
       \end{align}
   We want to show that 
   \begin{align}
   	M_S<1,  \text{ in } \MR\times V.  \label{20221107Target_AprioriEstimates}
   \end{align}
This implies that $\Phi(\tau, \ast)$ is strictly \SO-convex for any $\tau\in\partial\MR$. 
However, it's difficult to directly work with $M_S$, since it may not    be differentiable. We 
introduce the following approximation of $M_S$.
Let
\begin{align}
	\Qsp=\tr(K_S^p),   \label{defining_Qsp}\\
	Q^{[p]}_{S}=\left(\Qsp\right)^{\frac{1}{p}}.\label{defining_QsP}
\end{align} 
According to basic calculus, for $\lambda_1,\ ...\ , \lambda_n\geq 0$,
\begin{align}
	\lim_{p\rightarrow +\infty}\left(\lambda_1^p+\ ...\ + 
	\lambda_n^p\right)^{1/p}=\max\{\lambda_1,\ ...\ , \lambda_n\},
\end{align}
so
\begin{align}
	\lim_{p\rightarrow +\infty}\QsP=M_S.
\end{align}
       If we can show for $p$ big enough, 
       \begin{align}
       	\QsP\leq \max_{{\partial \MR\times V}} \QsP, \ \ \ \ \text{ in }\MR\times V,
       \end{align}
   then we can let $p$ go to $\infty$ and prove (\ref{20221107Target_AprioriEstimates}).
   
   In section \ref{sec:computation_for_Apriori_Estimate}, we prove that for the elliptic 
   operator $L=L^{\ijbar}\partial_{\ijbar}$ introduced in section 
  \ref{sec:Equation}, 
   \begin{align}
   	L^{\ijbar}\left(\Qsp\right)_{\ijbar}\geq 0,\ \ \ \ \ \text{ in }\MR\times V,
   \end{align}
providing $K_S\leq 1-\frac{1}{2p}$.
       In section \ref{sec:method_continuity}, using a continuity argument we prove 
       \begin{align}
       	\QsP\leq \max_{{\partial \MR\times V}} \QsP, \ \ \ \ \text{ in }\MR\times V.
       \end{align}
       In section \ref{sec:metriclowerBound}, by altering $S$, we prove a convexity estimate, which implies a metric lower bound estimate 
       \begin{align}
       	\omega_0+\iu\ddbar\Phi(\tau,\ast)\geq \mu\omega_0, 
       \end{align}
   for a constant $\mu>0$.
       
       In section \ref{sec:method_continuity} and \ref{sec:metriclowerBound}, we need 
       an apriori assumption that for any $\lambda\in [0,1]$, Problem 
       \ref{prob:NewPerturbation} has a solution $\Phi^\lambda$ with boundary value 
       $\lambda F$ and $\{\Phi^\lambda|\lambda\in[0,1]\}$ is a continuous curve in 
       $C^4(\overline{\MR}\times V)$ in $C^2$ topology.
       
       \subsection{Computation}\label{sec:computation_for_Apriori_Estimate}
       
       Suppose $\Phi$ is a $C^4$ solution to Problem \ref{prob:NewPerturbation}. In a 
       local coordinate chart, let
       \begin{align}
       	K&=B\Abarinverse \ \Bbar \Ainverse
       \end{align}
       and, for a positive integer $p$,
       \begin{align}
       	\Qp=\tr(K^p).
       \end{align}
       As matrices, $B, A$ and $K$ depends on the choice of coordinate, while $\Qp$ is 
       a well defined function on $\MR\times V$. In this section, we show
       \begin{align}
       	L^{\ijbar} \left(\Qp\right)_{\ijbar}\geq 0, \ \ \ \text{providing }K\leq 1-\frac{1}{2p}.
       \end{align}
   After proving this, we know for 
   \begin{align}
   	K_S=(B-S)\Abarinverse \ \overline{(B-S)} \Ainverse
   \end{align}
and 
\begin{align}
	\Qsp=\tr(K_S^p),
\end{align}
\begin{align}
 L^{\ijbar}\left(\Qsp\right)_{\ijbar}\geq 0,\ \ \ \ \text{providing } K_S\leq 1-\frac{1}{2p}. 
\end{align}
This is because $B-S, A$ satisfy the same set of equations as $B, A$ do. Similar to 
(\ref{eq:A_section2_first_Matrix}) (\ref{eq:B_section2_first_Matrix}), we have:
     \begin{align}
	L^{i\jbar}\partial_{i\jbar} A&=L^{i\jbar} (\partial_i A) A^{-1} \partial_\jbar A+ L^{\ijbar} 
\partial_\jbar 	(B-S) \overline{A^{-1}} \partial_i \overline {(B-S)};
	\label{eq:A_byBS_section3_Matrix}
	\\
	L^{i\jbar}\partial_{i\jbar} (B-S)&=L^{i\jbar} (\partial_i A) A^{-1} \partial_\jbar (B-S)+ 
	L^{\ijbar} 
	\partial_\jbar (B-S) \overline{A^{-1}} (\partial_i \overline A).
	\label{eq:BS_section3_Matrix}
\end{align}
Equations above are equivalent to 
(\ref{eq:A_section2_first_Matrix}) (\ref{eq:B_section2_first_Matrix}) because $S$ is a 
constant section and  all derivatives of $S$ are zero.
 
       Before  the computation of $L^{\ijbar}\left(\Qp\right)_{\ijbar}$, we do
       some preparation. First, we note that the conjugate of equation (\ref{eq:B_section2_first_Matrix}) is 
       equivalent to 
       \begin{align}
       	L^{\ijbar}\left(\Abarinverse\ \Bbar_i 
       	\Ainverse\right)_{\jbar}=0.\label{eq:B_equation_Simplified}
       \end{align}
       Then we introduce a quantity
        \begin{align}
        	\Bi=B_i-A_i\Ainverse B-B \Abarinverse \ \overline{A}_i. \label{eq:Bmi_definition}
        \end{align} The reason to introduce $\Bi$ is to combine some terms with $B_i$ to simplify the     computation.   Even $\Bi$ can be considered as a tensor, we just need to  consider    it as a  symmetric   matrix     valued     function    defined   in a 
    local coordinate chart. (\ref{eq:Bmi_definition}) is equivalent to 
       \begin{align}
       	A^{-1} \Bi \Abarinverse=\partial_i\left(\Ainverse B \Abarinverse\right). 
       	\label{eq:Bmi_definition_simplifed}
       \end{align}
   When $\Bi$ is differentiated by $L^{\ijbar}\partial_{\jbar}$, using 
   (\ref{eq:A_section2_first_Matrix}), (\ref{eq:B_section2_first_Matrix}) and the conjugate of (\ref{eq:A_section2_first_Matrix}), we find
   \begin{align}
   	L^{\ijbar}\partial_{\jbar}\Bi=\left(-\Bjbar \Abarinverse \ \Bbar_i \Ainverse B
   																  -B\Abarinverse \ \Bbar_i \Ainverse \Bjbar \right).
   																  \label{eq:L_act_on_B_i}
   \end{align}
       
       Now we start to compute $L^{\ijbar} \left(\Qp\right)_{\ijbar}$.  In the expression of 
       $\Qp$, we combine some terms together to simplify the computation:
       \begin{align}
       	\Qp&=\tr\left(B \Abarinverse \ \Bbar \Ainverse\right)^p\\
       	      &=\tr\left[(\Ainverse B \Abarinverse)\cdot \Bbar \cdot (\Ainverse B 
       	      \Abarinverse) \cdot \Bbar \cdot ... \cdot (\Ainverse B \Abarinverse) \cdot 
       	      \Bbar\right]. \label{20221107_328}
       \end{align}
        In (\ref{20221107_328}), $ (\Ainverse B \Abarinverse)$  and 
        $\overline B$ both appear $p$ times. When $\partial_i$ act on any of $(\Ainverse B 
        \Abarinverse)$ (or $\overline B$), we get the same result. So
        \begin{align}
        	\partial_i\Qp=&p\cdot\tr\left[\partial_i(\Ainverse B \Abarinverse)\cdot \Bbar \cdot 
        	(\Ainverse B   	\Abarinverse) \cdot \Bbar \cdot ... \cdot (\Ainverse B 
        	\Abarinverse) \cdot 
        	\Bbar\right]  \label{20221107_329}\\
        	+&p\cdot\tr\left[(\Ainverse B \Abarinverse)\cdot \partial_i\Bbar \cdot 
        	(\Ainverse B   	\Abarinverse) \cdot \Bbar \cdot ... \cdot (\Ainverse B 
        	\Abarinverse) \cdot 
        	\Bbar\right].
        \end{align}
        We plug (\ref{eq:Bmi_definition_simplifed}) into (\ref{20221107_329}) and get
                \begin{align}
        	\partial_i\Qp=&p\cdot\tr\left[\Ainverse \Bi \Abarinverse\cdot \Bbar \cdot 
        	(\Ainverse B   	\Abarinverse) \cdot \Bbar \cdot ... \cdot (\Ainverse B 
        	\Abarinverse) \cdot 
        	\Bbar\right]  \label{20221107_331}\\
        	+&p\cdot\tr\left[(\Ainverse B \Abarinverse)\cdot \partial_i\Bbar \cdot 
        	(\Ainverse B   	\Abarinverse) \cdot \Bbar \cdot ... \cdot (\Ainverse B 
        	\Abarinverse) \cdot 
        	\Bbar\right]. \label{20221107_332}
        \end{align}
   We reorganize terms in the product of (\ref{20221107_331}) 
    (\ref{20221107_332}):
       \begin{align}
       	\partial_i\Qp=&p\cdot\tr\left[ \Bi\cdot (\Abarinverse\  \Bbar \Ainverse)\cdot B  \cdot
       	(\Abarinverse \ \overline{B} \Ainverse)\cdot  ... \cdot B  \cdot (\Abarinverse \ 
       	\overline{B} \Ainverse)\right]  
       	\label{20221107_333}\\
       	+&p\cdot\tr\left[(\Abarinverse \partial_i \Bbar \Ainverse)\cdot B\cdot (\Abarinverse 
       	\ \overline{B} \Ainverse)\cdot B\cdot ...\cdot     	(\Abarinverse \ \overline{B} 
       	\Ainverse) \cdot 
       	B\right]. \label{20221107_334}
       \end{align}
       When apply $L^{\ijbar}\partial_{\jbar}$ to $\Qp_i$, $L^{\ijbar}\partial_\jbar$ acts on  
       6 kinds of terms:
       \begin{itemize}
       	\item [(i)] $L^{\ijbar}\partial_\jbar$ acts on $\Bi$ in (\ref{20221107_333});
       	\item [(ii)] $L^{\ijbar}\partial_\jbar$ acts on $(\Abarinverse\  \Bbar \Ainverse)$ in 
       	(\ref{20221107_333});
       	\item [(iii)] $L^{\ijbar}\partial_\jbar$ acts on $B$ in (\ref{20221107_333});
       	\item [(iv)] $L^{\ijbar}\partial_\jbar$ acts on $(\Abarinverse \partial_i \Bbar 
       	\Ainverse)$ in (\ref{20221107_334});
       	\item [(v)]  $L^{\ijbar}\partial_\jbar$ acts on $B$ in (\ref{20221107_334});
       	\item [(vi)]  $L^{\ijbar}\partial_\jbar$ acts on $(\Abarinverse \ \overline{B} 
       	\Ainverse)$ in 
       	(\ref{20221107_334}).
       \end{itemize}
       In the following we do the computation separately.
       
       (i) When $L^{\ijbar}\partial_\jbar$ acts on $\Bi$ in (\ref{20221107_333}), the result 
       is 
       \begin{align}
       	&p\cdot \tr\left[L^{\ijbar}\partial_\jbar\Bi(\Abarinverse \ \overline{B}  	
      \Ainverse)\cdot\left(B\cdot (\Abarinverse \ \overline{B} \Ainverse)\right)^{p-1} 
      \right]L^{\ijbar}
      \\
     =&p\cdot \tr\left[-B_\jbar \Abarinverse  \ \Bbar_i \Ainverse\cdot\left(B\cdot 
     (\Abarinverse \ 
     \overline{B} 
     \Ainverse)\right)^{p} 
     \right]L^{\ijbar}
     \\
     +&p\cdot \tr\left[-\Abarinverse\  \Bbari \Ainverse B_\jbar\cdot\left( (\Abarinverse \ 
     \overline{B} 
     \Ainverse)\cdot 
     B\right)^{p} 
     \right] L^{\ijbar}.
     \end{align}
       To get this, we need to use equation (\ref{eq:L_act_on_B_i}).
     
      (ii) When $L^{\ijbar}\partial_\jbar$ acts on $(\Abarinverse\  \Bbar \Ainverse)$ in 
       (\ref{20221107_333}), we need to use the 
       conjugate of (\ref{eq:Bmi_definition_simplifed}):
          \begin{align}\label{20221107_338}
          	\Abarinverse\ \overline{\Bj} \Ainverse=\partial_{\jbar}(\Abarinverse \ \overline{B} 
          	\Ainverse). 
          \end{align}The result is the sum of $p$ terms:
      \begin{align}
      	&p\cdot \tr\left[\Bi
      	  (\Abarinverse \ \overline{B} \Ainverse B)^0
      	   \Abarinverse \ \overline{\Bj} \Ainverse (B \Abarinverse \ \Bbar 
      	\Ainverse)^{p-1}\right]L^{\ijbar}\\
      	+&p\cdot \tr\left[\Bi
      	 (\Abarinverse \ \overline{B} \Ainverse B)^1
      	 \Abarinverse \ 
      	\overline{\Bj} \Ainverse (B \Abarinverse \ \Bbar 
      	\Ainverse)^{p-2}\right]L^{\ijbar}\\
      	&\ \ \ ...  \nonumber\\
      +&	p\cdot \tr\left[\Bi (\Abarinverse \ \overline{B} \Ainverse B)^{p-1} \Abarinverse \ 
      \overline{\Bj} \Ainverse (B \Abarinverse \ \Bbar 
      \Ainverse)^0\right]L^{\ijbar}.
      \end{align}
       In above, fictitious terms $(B \Abarinverse \ \Bbar 
       \Ainverse)^0$ and $ (\Abarinverse \ \overline{B} \Ainverse B)^0$ are added in to make the pattern clearer. We will do this in other parts of the computation either.

      (iii)  When $L^{\ijbar}\partial_\jbar$ acts on $B$ in (\ref{20221107_333}), we don't 
      need to use other equations. The result is the sum of $p-1$ terms:
      \begin{align}
      	&p\cdot \tr\left[\Bi(\Abarinverse \ \overline{B} \Ainverse) 
      		(B \Abarinverse \ \Bbar \Ainverse)^{0}
      	\Bjbar 
      	(\Abarinverse \ \overline{B} \Ainverse)  (B \Abarinverse \ \Bbar 
      	\Ainverse)^{p-2}\right]L^{\ijbar}\\
      	+&p\cdot \tr\left[\Bi(\Abarinverse \ \overline{B} \Ainverse) 
      		(B \Abarinverse \ \Bbar \Ainverse)^1
      		\Bjbar 
      	(\Abarinverse \ \overline{B} \Ainverse)  (B \Abarinverse \ \Bbar 
      	\Ainverse)^{p-3}\right]L^{\ijbar}\\
      	&\ \ \ \ ...\nonumber\\
      	+&p\cdot \tr\left[\Bi(\Abarinverse \ \overline{B} \Ainverse) (B \Abarinverse \ \Bbar 
      	\Ainverse)^{p-2} \Bjbar 
      	(\Abarinverse \ \overline{B} \Ainverse)(B \Abarinverse \ \Bbar 
      	\Ainverse)^0 \right]L^{\ijbar}.
      \end{align}
     
      (iv) When  $L^{\ijbar}\partial_\jbar$ acts on $(\Abarinverse \partial_i \Bbar 
       \Ainverse)$ in (\ref{20221107_334}), the result is zero. This is because of the 
       equation (\ref{eq:B_equation_Simplified}).
             
      (v)  When  $L^{\ijbar}\partial_\jbar$ acts on $B$ in (\ref{20221107_334}), we 
      don't need to use other equations. Simply differentiating $B$, we get the following
      result which
      is the sum of $p$ terms:
         \begin{align}
      	&p\cdot \tr\left[\Abarinverse \partial_i \Bbar \Ainverse (B \Abarinverse \ \Bbar 
      	\Ainverse)^0 \Bjbar 
      	(\Abarinverse \ \overline{B} \Ainverse B)^{p-1}\right]L^{\ijbar}\\
      	+&p\cdot \tr\left[\Abarinverse \partial_i \Bbar \Ainverse(B \Abarinverse \ \Bbar 
      	\Ainverse)^1\Bjbar 
      	(\Abarinverse \ \overline{B} \Ainverse B)^{p-2}\right]L^{\ijbar}\\
      	&\ \ \ \ ...\nonumber\\
      	+&p\cdot \tr\left[\Abarinverse \partial_i \Bbar \Ainverse(B \Abarinverse \ \Bbar 
      	\Ainverse)^{p-1}\Bjbar (\Abarinverse \ \overline{B} \Ainverse B)^0 \right]L^{\ijbar}.
      \end{align}
     
      (vi)  When $L^{\ijbar}\partial_\jbar$ acts on $(\Abarinverse \ \overline{B} 
       \Ainverse)$ in  (\ref{20221107_334}), we need to use  (\ref{20221107_338}). The 
       result is the sum of $p-1$ terms 
                \begin{align}
       	&p\cdot \tr\left[(\Abarinverse \partial_i \Bbar \Ainverse) B (\Abarinverse \ \overline{B} \Ainverse B)^0
       	(\Abarinverse \ \overline{\Bj} \Ainverse) B (\Abarinverse \ \overline{B} \Ainverse 
       	B)^{p-2}\right]L^{\ijbar}\\
       	+&p\cdot \tr\left[(\Abarinverse \partial_i \Bbar \Ainverse) 
       	B(\Abarinverse \ \overline{B} \Ainverse B)^1(\Abarinverse \ \overline{\Bj} \Ainverse)
       	B(\Abarinverse \ \overline{B} \Ainverse B)^{p-3}\right]L^{\ijbar}\\
       	&\ \ \ \ ...\nonumber\\
       	+&p\cdot \tr\left[(\Abarinverse \partial_i \Bbar \Ainverse) B 
       	(\Abarinverse \ \overline{B} \Ainverse B)^{p-2}(\Abarinverse \ \overline{\Bj} 
       	\Ainverse) B (\Abarinverse \ \overline{B} \Ainverse B)^0\right]L^{\ijbar}.
       \end{align}
   
   At a point $(\tau_0, z_0)\in\MR\times V$, we change coordinate on $V$ to diagonalize 
   $A, B$. As discussed in section \ref{sec:more_notation}, we need to find $P$, so 
   that 
   \begin{align}
   	PAP^\ast=I,\ \ \ PBP^T=\Lambda=\diag(\Lambda_1,\ ...\ , \Lambda_n)\geq 0. 
   \end{align}This can be done by Lemma 
\ref{lemma:simultaneous_diagonalization}.
Denote that 
\begin{align}
	\Bi=(\MB_{i;\alpha\beta}), \ \ \ \ \partial_i \Bbar=(\overline B_{i; \alpha\beta}).
\end{align}
With these notation, results of (i)-(vi) can be simplified:
Result of (i) is:
\begin{align}
	p\left(-\Bbariab \overline{\Bbarjab}( \Lambda_\alpha^{2p}+ 
	\Lambda_\beta^{2p})\right)L^{\ijbar};
\end{align}  	result of (ii) is
\begin{align}
	p\left(\mBiab\overline{\mBjab}(\Lambda_\alpha^{2p-2}+\Lambda_\alpha^{2p-4}
	\Lambda_\beta^{2}+\ ...\ +\Lambda_\beta^{2p-2})\right)L^{\ijbar};
\end{align}
	result of (iii) is
\begin{align}
	p\left(\mBiab\overline{\Bbarjab}(\Lambda_\alpha^{2p-3}
	\Lambda_\beta+\Lambda_\alpha^{2p-5}
	\Lambda_\beta^{3}\ ...\ +\Lambda_\alpha
	\Lambda_\beta^{2p-3})\right)L^{\ijbar}; 
\end{align}result of (iv) is still $0$; 
result of (v) is
\begin{align}
	p\left(\Bbariab\overline{\Bbarjab}(\Lambda_\alpha^{2p-2}+\Lambda_\alpha^{2p-4}
	\Lambda_\beta^{2}+\ ...\ +\Lambda_\beta^{2p-2})\right)L^{\ijbar};
\end{align}
result of (vi) is
\begin{align}
	p\left(\Bbariab\overline{\mBjab}(\Lambda_\alpha^{2p-3}
	\Lambda_\beta+\Lambda_\alpha^{2p-5}
	\Lambda_\beta^3\ ...\ +\Lambda_\alpha
	\Lambda_\beta^{2p-3})\right)L^{\ijbar}.
\end{align}

Summing up results of (i)-(vi), we get
\begin{align}
	L^{\ijbar}\left(\Qp\right)_{\ijbar}=p\sum_{i,j}\sum_{\alpha,\beta} L^{\ijbar}(\mBiab, 
	\Bbariab)\mathcal{W}_{\alpha\beta}\overline{\left(\begin{array}{cc}
		\mBjab\\ \Bbarjab
	\end{array}\right)},   \label{eq:LQp_matrix_expression}
\end{align}
where
\begin{align}
	\mathcal{W}_{\alpha\beta}=\left(
	\begin{array}{cc}
		\sum_{k=0}^{p-1}\Lambda_\alpha^{2k} \Lambda_\beta^{2p-2-2k}
  		&\sum_{k=0}^{p-2}\Lambda_\alpha^{2k+1}\Lambda_\beta^{2p-3-2k}\\
  		\sum_{k=0}^{p-2}\Lambda_\alpha^{2p-3-2k}\Lambda_\beta^{2k+1}
  		&\sum_{k=0}^{p-1}\Lambda_\alpha^{2k} \Lambda_\beta^{2p-2-2k}
  		-\Lambda_\alpha^{2p}	-\Lambda_\beta^{2p}
	\end{array}
	\right).\label{eq:Wab_expression}
\end{align}
It's obvious that when $\mathcal{W}_{\alpha\beta}\geq 0$, 
(\ref{eq:LQp_matrix_expression})$\geq0$.
In the following, we show when $\Lambda_\alpha, 
\Lambda_\beta\leq \sqrt{1-\frac{1}{2p}}$, 
$\mathcal{W}_{\alpha\beta}\geq 0.$ 

According to linear algebra, 
$\mathcal{W}_{\alpha\beta}\geq 0$ if and only if $\tr (\mathcal{W}_{\alpha\beta})\geq 
0$ 
and $\det (\mathcal{W}_{\alpha\beta})\geq 0$. $\tr (\mathcal{W}_{\alpha\beta})$ is easy 
to 
compute:
\begin{align}
\tr	 (\mathcal{W}_{\alpha\beta})\geq 
\Lambda_\alpha^{2p-2}
+\Lambda_\beta^{2p-2}
-\Lambda_\alpha^{2p}
-\Lambda_\beta^{2p}.
\end{align}It's greater than $0$ providing
 \begin{align}
\Lambda_\alpha, 
\Lambda_\beta<1.\label{eq:restiction0}
\end{align} To investigate the sign of $\det(\mathcal{W}_{\alpha\beta})$, we 
need to simplify the right-hand side of (\ref{eq:Wab_expression}). When 
$\Lambda_\alpha=\Lambda_\beta$, denoting
$\Lambda_\alpha=\Lambda_\beta=\lambda$, 
\begin{align}
	\mathcal{W}_{\alpha\beta}=\left(
	\begin{array}{cc}
		p\lambda^{2p-2}& (p-1) \lambda^{2p-2}\\
		(p-1) \lambda^{2p-2}& p \lambda^{2p-2}-2\lambda^{2p}
	\end{array}
	\right).
\end{align}
The determinant follows by direct computation:
\begin{align}
	\det(\mathcal{W}_{\alpha\beta})
	=(2p-1)\lambda^{4p-4}\left[1-\frac{2p}{2p-1}\lambda^2\right].  
\end{align}
It's non-negative proving 
\begin{align}
	\lambda^2\leq 1-\frac{1}{2p}.   \label{eq:restiction1}
\end{align}
When $\Lambda_\alpha\neq \Lambda_\beta$, we use the summation formula for geometric series to compute the summation in (\ref{eq:Wab_expression}).
 We assume $\Lambda_\alpha^2>\Lambda_\beta^2$. Note that we have $\Lambda_\alpha, \Lambda_\beta\geq 0$, so when $\Lambda_\alpha\neq \Lambda_\beta$, $\Lambda^2_\alpha\neq \Lambda^2_\beta$. The result is
 \begin{align}
 	\mathcal{W}_{\alpha\beta}=
 	\left(
 	\begin{array}{cc}
 		\frac{\Lambda_\beta^{2p}-\Lambda_\alpha^{2p}}
 		           {\Lambda^2_\beta-\Lambda^2_\alpha} &
 		-\frac{\Lambda_\beta \Lambda_\alpha^{2p-1}
 	     	-\Lambda_\beta^{2p-1}\Lambda_\alpha}
               		{\Lambda^2_\beta-\Lambda^2_\alpha}	\\      
               		-\frac{\Lambda_\beta \Lambda_\alpha^{2p-1}
               		-\Lambda_\beta^{2p-1}\Lambda_\alpha}
               	{\Lambda^2_\beta-\Lambda^2_\alpha}&
               		\frac{\Lambda_\beta^{2p}-\Lambda_\alpha^{2p}}
               		{\Lambda^2_\beta-\Lambda^2_\alpha}      
               		-\Lambda_\alpha^{2p}-\Lambda_\beta ^{2p}
 	\end{array}
 	\right).
 \end{align}
The determinant follows by straightforward computation:
\begin{align}
\det(	\mathcal{W}_{\alpha\beta})
	=\frac{	(-\Lambda^{4p}_\beta+\Lambda^{4p-2}_\beta)-(-\Lambda^{4p}_\alpha+\Lambda^{4p-2}_\alpha)
				}
				{\Lambda^2_\beta-\Lambda^2_\alpha}.  \label{20221107_366}
\end{align}
Let $f(x)=-x^{2p}+x^{2p-1}.$ Then (\ref{20221107_366}) can be simplified as
\begin{align}
\det(	\mathcal{W}_{\alpha\beta})=
	\frac{f(\Lambda_\beta^2)-f(\Lambda_\alpha^2)}
	{\Lambda^2_\beta-\Lambda^2_\alpha}.
\end{align}
It's non-negative, providing $\Lambda_\alpha^2$ and $ \Lambda_\beta^2$ stay on an 
interval where $f$ is non-decreasing.  By computing the derivative of $f$ we know $f$ 
is non-decreasing on $[0,1-\frac{1}{2p}]$. Combining with (\ref{eq:restiction0}) 
(\ref{eq:restiction1}), we know $\det(\mathcal{W}_{\alpha\beta})\geq 0$, providing 
$\Lambda_\alpha^2, \Lambda_\beta^2\leq 1-\frac{1}{2p}.$ Therefore, 
\begin{align}
	L^{\ijbar}\left(\Qp\right)_{\ijbar}\geq 0, \ \ \ \ \text{ providing } K\leq 1-\frac{1}{2p}.
	\label{eq:consequence_LQPgeq0_section3_1}
\end{align}
In the computation above, we can replace $B$ by $B-S$ and get 
\begin{align}
	L^{\ijbar}\left(\Qsp\right)_{\ijbar}\geq 0, \ \ \ \ \text{ providing } K_S\leq 1-\frac{1}{2p}.
\end{align}

As a result, we have the following proposition:
\begin{proposition}
	\label{prop:computation}
	Suppose $\Phi$ is a $C^4$ solution to Problem \ref{prob:NewPerturbation} and $S$ is a constant section of $T_{2,0}^{\ast}(V)$. For $L, \ K_S,\ \Qsp$ and  $\ \QsP$,  defined by (\ref{eq:L_introduced})	(\ref{defining_KS}) 
	 (\ref{defining_Qsp}) 
	(\ref{defining_QsP}) respectively, we have 
	\begin{align}
		L^{\ijbar}\partial_{\ijbar}\left(\Qsp\right)\geq 0, 
	\end{align} 
providing $ K_S\leq 1-\frac{1}{2p},$
and, as a consequence,
\begin{align}
	\QsP\leq \max_{{\partial \MR\times V}} \QsP, \ \ \ \ \ \ \text{ in }\MR\times V,
\end{align}providing $ K_S\leq 1-\frac{1}{2p},$ in $\MR\times V$.
\end{proposition}
            \begin{remark}
            	Let
            	\begin{align}
            		Q^{[\rho, p]}_S=\left(\tr(K_S^{\rho p})\right)^{\frac{1}{p}}.
            	\end{align}
            	With more complicated computations we can show
            	\begin{align}
            	 L^{\ijbar} \left(		Q^{[\rho, p]}_S\right)_{\ijbar}\geq 0,
            	\end{align}
            providing $K_S\leq1- \frac{1}{2\rho}$. 
            So, by letting $p$ go to $\infty$, we know
            \begin{align}
            	L^{\ijbar}\left(M_S^\rho\right)_{\ijbar}\geq 0,
            \end{align}
        providing $K_S\leq1- \frac{1}{2\rho}$,
        in the sense of  viscosity solution (see section 6 of \cite{UserGuide}). However, 
        the  current    result    (\ref{eq:consequence_LQPgeq0_section3_1}) is enough for 
        our use.  We can also consider 
        \begin{align}
        	\tr\left(e^{p K_S}\right)   \label{mentioned_372}
        \end{align} and achieve a similar result. In \cite{HuC2Perturb}, (\ref{mentioned_372}) is used in $n=1$ case. 
            \end{remark}
             \subsection{Preservation of \SO-Convexity by  the Method of Continuity}
             \label{sec:method_continuity}
             In addition to (\ref{20221107BoundarySOmega_0Convexity}), in this section,  we make the following assumption.
             \begin{assumption} 
             	\label{assumption:CurveC2}
             For any $\sigma\in[0,1]$, Problem \ref{prob:NewPerturbation}
             with boundary value $\sigma F$ has a solution 
             $\Phi^\sigma$ and $\{\Phi^{\sigma}|\sigma\in[0,1]\}$ is a continuous curve in 
             $C^4(\overline{\MR}\times V)$ in $C^2$ topology.
             \end{assumption}
         \begin{remark}
         	According to the ellipticity of equation (\ref{eq:new_Perturbation_in_Inverse_Form_Section1}), which is proved in section \ref{sec:PhiX_Phi_XX_Equations}, the solution to Problem \ref{prob:NewPerturbation} is unique.  As a consequence, $\Phi^0$ must equal to $0$. 
         \end{remark}
             Let
             \begin{align}
             	A_\sigma&=(\Phi^{\sigma}_{\alpha\betabar}+b_{{\alpha\betabar}}),\\
             	B_{S,\sigma}&=(\Phi^{\sigma}_{\alpha\beta}
             									+\sigma S_{{\alpha\beta}}),\\
             	K_{S, \sigma}&=B_{S,\sigma}\overline{A_\sigma^{-1}}\ 
             									\overline{B_{S,\sigma}} A_\sigma^{-1},\\
             	Q^{<p>}_{S, \sigma}&=\tr (K_{S,\sigma}^p),\\
             	Q^{[p]}_{S,\sigma}&=\left(Q^{<p>}_{S, \sigma}\right)^{\frac{1}{p}},
             \end{align}
             and for any $(\tau, z)\in \MR\times V$,
             \begin{align}
             	M_{S, \sigma}(\tau, z)=\text{Maximum Eigenvalue of } K_{S, \sigma}(\tau, z).
             \end{align}
         
             According to the assumption (\ref{20221107BoundarySOmega_0Convexity}), 
             \begin{align}
             	\max_{\partial \MR\times V} M_{S, 1}<1.
             \end{align}
         So we can choose $p$ large enough, such that
         \begin{align}
         	\max_{\partial \MR\times V} Q^{[p]}_{S,1}<1-\frac{1}{2p}.
         \end{align}
             This can be done because when $p\rightarrow +\infty$,
             \begin{align}
             	1-\frac{1}{2p}\rightarrow 1
             \end{align}
             and
             \begin{align}
             	\max_{\partial \MR\times V} Q^{[p]}_{S,1}
             	\rightarrow \max_{\partial \MR\times 	V} M_{S, 1}<1. 
             \end{align}
             According to Lemma \ref{lemma:monotone}, for  any $(\tau, z) \in \partial\MR\times V$, $Q^{[p]}_{S,\sigma}(\tau, z)$ is a 
             monotone        non-decreasing    function of     $\sigma$, so 
             $\max_{\partial \MR\times V} Q^{[p]}_{S,\sigma}$ is a monotone non-decreasing function of $\sigma$. Thus
             \begin{align}
             	\max_{\partial \MR\times V} Q^{[p]}_{S,\sigma}<1-\frac{1}{2p},
             \end{align}
             for any $\sigma\in[0,1]$. Therefore, for no $\sigma \in[0,1]$, 
             \begin{align}
             	\max_{ \cMR\times V} Q^{[p]}_{S,\sigma}\in
             	\left(\max_{\partial \MR\times V} Q^{[p]}_{S,\sigma}, 1-\frac{1}{2p}\right) . \label{383_invalid}
             \end{align} 
             This is because if  (\ref{383_invalid}) is valid, for $\sigma=\sigma_0$, then
             \begin{align}
             	\max_{\cMR\times V}Q^{[p]}_{S,\sigma_0}\leq 1-\frac{1}{2p}, 
             \end{align} and, as a consequence,
         \begin{align}
         	\max_{\cMR\times V} M_{S, \sigma_0} \leq 1-\frac{1}{2p}.
         \end{align}
So, using Proposition \ref{prop:computation}, we know
     \begin{align}
     	\max_{\cMR\times V} Q^{[p]}_{S,\sigma_0}\leq
     	 \max_{\partial \MR\times V} Q^{[p]}_{S,\sigma_0},
     \end{align} which contradicts with (\ref{383_invalid}).
 
 The function $\max_{\cMR\times V} Q^{[p]}_{S,\sigma}$ is a continuous 
     function of $\sigma$, because  of Assumption \ref{assumption:CurveC2}.  And $\max_{\cMR\times V} 
     Q^{[p]}_{S,0}=0$, according to the uniqueness of $C^2$ solution. So  
     \begin{align}
     	\max_{\cMR\times V} Q^{[p]}_{S,\sigma}\leq \max_{\partial \MR\times V} 
     																							Q^{[p]}_{S,\sigma}
     																							, \ \ \ \ \text{for any }\sigma\in[0,1]. 
     \end{align}  
 
    \begin{figure}[h]
     \centering  
     \includegraphics[height=5.5cm]{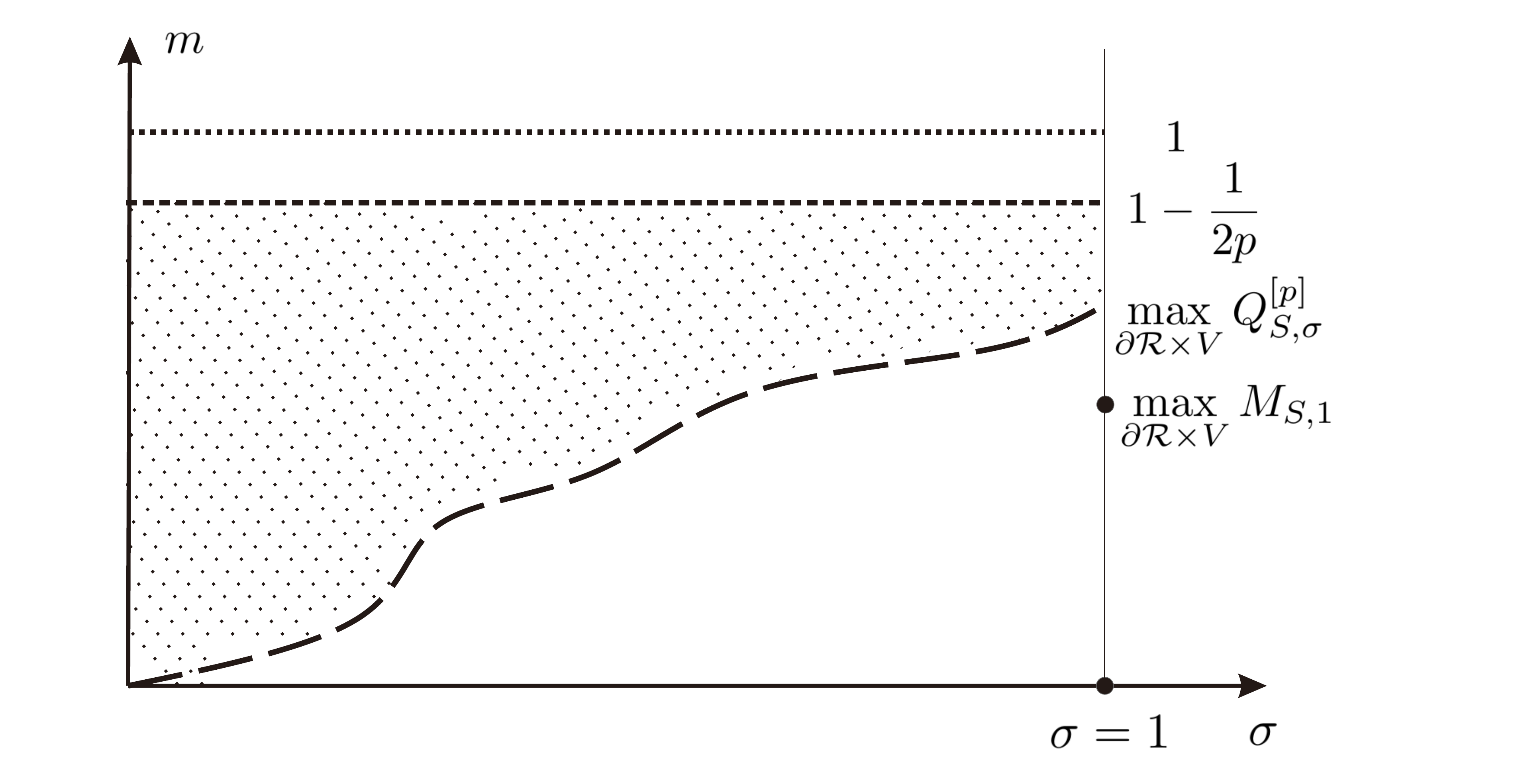}
     \caption{Method of Continuity}
     \label{fig:Method_of_Continuity}
 \end{figure}
 As illustrated by Figure \ref{fig:Method_of_Continuity}, 
  the dashed curve
 \begin{align}
 	\left\{	(\sigma, \max_{\partial\MR\times V} Q^{[p]}_{S,\sigma})\big| \sigma\in[0,1]\right\}
 \end{align} raises up from left to right and stays below the line $\{m=1-\frac{1}{2p}\}$. The continuous curve
 \begin{align}
 \left\{	(\sigma, \max_{\cMR\times V} Q^{[p]}_{S,\sigma})\big| \sigma\in[0,1]\right\},
 \end{align} whose left endpoint is $(0,0)$,  cannot intersect with the shadowed area, so it has to stay below the shadowed area.  Actually, it has to coincide with the dashed curve.

Therefore, we have, for any $\sigma\in[0,1]$,
 \begin{align}
 	M_{S, \sigma}\leq  \max_{\overline{\MR}\times V} Q^{[p]}_{S,\sigma}
 =
 	 \max_{\partial\MR\times V} Q^{[p]}_{S,\sigma}\leq  \max_{\partial\MR\times V} Q^{[p]}_{S,1}, \ \ \ \ \ \ 
 	\text{ in }\MR\times V.
 \end{align}
Let $p\rightarrow 0$, we get
 \begin{align}
	M_{S, \sigma}\leq \max_{\partial \MR\times V}M_{S,1}, \ \ \ \ \ \ 
	\text{ in }\MR\times V.
\end{align}

In sum, we have the following Proposition.

\begin{proposition}
	\label{prop:Preservation_Semi_Convexity_by_Continuity_Method}
	Suppose $S$ is a constant section of $T_{2,0}^{\ast}(V)$ and  $F\in C^2({\partial \MR\times V})$ satisfies that
	\begin{align}
		F(\tau, \ast) \text{ is strictly \SO-convex, for any $\tau\in\partial\MR$.}
	\end{align}
In addition, we assume that the Assumption \ref{assumption:CurveC2} is satisfied. Then the solution $\Phi$ to Problem \ref{prob:NewPerturbation}, with boundary value $F$, satisfies
\begin{align}
	\Phi(\tau, \ast) \text{ is strictly \SO-convex, for any $\tau\in\MR$.}
\end{align}
\end{proposition}
             \subsection{Convexity and Metric Lower Bound Estimates by Altering $S$}   
             \label{sec:metriclowerBound}
            
            
            In this section, we assume that the conditions of Proposition \ref{prop:Preservation_Semi_Convexity_by_Continuity_Method} are satisfied. In addition, we assume that, for a constant $\delta>0$, 
            \begin{align}
            	F(\tau, \ast) \text{ is \SO-convex of modulus $>\delta$,  for any $\tau\in\partial \MR$}, \label{condition:Convex_of_Module_delta}
            \end{align}
        or equivalently, according to Lemma \ref{lemma:Equivalence Between Module of Convexity and Degree of Convexity}, 
              \begin{align}
            	F(\tau, \ast) \text{ is \SO-convex of degree $>\delta$,  for any $\tau\in\partial \MR$}. \label{condition:Convex_of_degree_delta}
            \end{align}
            
            Condition (\ref{condition:Convex_of_degree_delta}) says, for any constant section $\Theta$ of $T_{2,0}^\ast(V)$, with 
            \begin{align}
            	\Theta \overline{W^{-1}}\ \Thetabar W^{-1}\leq \delta^2,   \label{3102_1119}
            \end{align}we have
             \begin{align}
            	F(\tau, \ast) \text{ is strictly $(S+\Theta,\omega_0)$-convex,  for any $\tau\in\partial \MR$}.
            \end{align}
        Here $W=(b_{\alpha\betabar})$.
         So, by Proposition \ref{prop:Preservation_Semi_Convexity_by_Continuity_Method}, for any constant section $\Theta$ of $T_{2,0}^\ast(V)$, satisfying (\ref{3102_1119}), we have 
      \begin{align}
      \Phi(\tau, \ast) \text{ is strictly $(S+\Theta,\omega_0)$-convex,  for any $\tau\in \MR$}.
      \end{align}
 Therefore,
    \begin{align}
 	\Phi(\tau, \ast) \text{ is \SO-convex of degree $>\delta$,  for any $\tau\in \MR$}. \label{3104_1119}
 \end{align}
By Lemma \ref{lemma:Equivalence Between Module of Convexity and Degree of Convexity}, we know  (\ref{3104_1119}) is equivalent to 
    \begin{align}
	\Phi(\tau, \ast) \text{ is \SO-convex of modulus $>\delta$,  for any $\tau\in \MR$}. \label{3105_1119}
\end{align}
Then, by the definition of modulus of convexity, Definition \ref{def:Module_SOmegaZeroConvexityC0_1117}, we know
\begin{align}
	\omega_0+\iu\ddbar\Phi(\tau, \ast)> \delta \omega_0, \ \ \ \ \ \ \text{for any }\tau\in\MR.
\end{align}

                        As a result, we have the following convexity and metric lower bound estimate:
                         \begin{proposition}
                         	[Apriori Convexity and Metric Lower Bound Estimate]
                         	\label{prop:metric_lower_bound}
                         	Suppose that, for a constant $\delta>0$ and a constant  section $S$ of $T_{2,0}^{\ast}(V)$, $F\in C^{\infty}({\partial \MR\times V})$ satisfies
                         	 \begin{align}
                         		F(\tau, \ast) \text{ is \SO-convex of degree $>\delta$,  for any $\tau\in\partial \MR$}.
                         	\end{align}
                         	In addition, Assumption \ref{assumption:CurveC2} is satisfied. Then a solution $\Phi$ to Problem \ref{prob:NewPerturbation}, with boundary value $F$ satisfies
                         	\begin{align}
                         		\Phi(\tau, \ast) \text{ is \SO-convex of degree $>\delta$,  for any $\tau\in \MR$}
                         	\end{align}
                     and, as a consequence,
                     	\begin{align}
                     \omega_0+\iu\ddbar	\Phi(\tau, \ast)>\delta\omega_0 \text{, \ \ \ \ \ \  for any $\tau\in \MR$}. 
                     \end{align}
                         \end{proposition}

\section{Existence of Solutions to the Perturbed Equation}
            \label{sec:Existence}
        In this section, we prove the existence of smooth solutions to Problem 
        \ref{prob:NewPerturbation}.

        In section \ref{sec:PhiX_Phi_XX_Equations}, we discuss some basic properties of 
        equation (\ref{eq:new_Perturbation_in_Inverse_Form_Section1}), including 
        ellipticity and concavity. Then  in section 
        \ref{sec:Torus_Partial_C2Estimates},  we derive a  directional partial $C^2$         estimate, in the direction of the affine manifold. The estimate allows us to derive  $C^0$ and $C^1$ estimates, in section 
        \ref{sec:C0C1}.  Then with the metric lower bound estimate, Proposition 
        \ref{prop:metric_lower_bound}, we prove $C^2$ and $C^{2,\alpha}$ estimates in 
        section \ref{sec:C2} and \ref{sec:C2alpha}. Finally, in section 
        \ref{sec:existence_by_method_continuity}, we prove the existence of smooth 
        solutions.

        
        \subsection{Basic Properties of the Elliptic Perturbation Equation}
        \label{sec:PhiX_Phi_XX_Equations}
        First,   equation 
        (\ref{eq:new_Perturbation_in_Inverse_Form_Section1}) is elliptic.  
        This has been indicated by (\ref{eq:D_theta_NewPerturbation}). More precisely, let $\Phi^\lambda$ be a family of solutions of equation (\ref{eq:new_Perturbation_in_Inverse_Form_Section1}) with $\Phi^0=\Phi$
        and $\frac{d}{d\lambda}\Phi^\lambda=\Psi$, at $\lambda=0$. Then differentiating 
        \begin{align}			\label{equation_lambda_family}
        	\Phi^\lambda_{\tautaubar}-\Phi^{\lambda}_{\tau\betabar} g^{{\alpha\betabar}}_\lambda\Phi^{\lambda}_{\taubar\alpha}
        	=\epsilon b_{{\alpha\betabar}} g^{{\alpha\betabar}}_\lambda
        \end{align}
    with respect to $\lambda$, at $\lambda=0$, gives
    \begin{align}
    	L^{\ijbar}\Psi_{\ijbar}=0,
    \end{align} 
        where $L^{\ijbar}$ was introduced in section \ref{sec:Equation} by (\ref{eq:L_introduced}). In (\ref{equation_lambda_family}), $g_\lambda^{{\alpha\betabar}}$ is the inverse of $b_{{\alpha\betabar}}+\Phi^\lambda_{{\alpha\betabar}}$.

        Then we consider the concavity. We will show 
        \begin{align}		
        	F(\Phi_{\ijbar})=\log(\Phi_{\tau\taubar}-\Phi_{\tau\betabar} g^{\abbar} 
        	\Phi_{\alpha\taubar})-\log(b_{\abbar}g^{\abbar}) \label{concave_Function}
        \end{align}is a concave function of $\Phi_{\ijbar}$, providing 
    \begin{align}
    	\left(
    \begin{array}{cc}
    	\Phi_{\tau\taubar}& \Phi_{\tau\betabar}\\
    	\Phi_{\alpha\taubar}& \Phi_{{\alpha\betabar}}+b_{\alpha\betabar}
    \end{array}
    \right)>0.
    \end{align}
    Actually, if we denote 
      \begin{align}\label{concave_Function_1}
     	F_1(\Phi_{\ijbar})&=\log(\Phi_{\tau\taubar}-\Phi_{\tau\betabar} g^{\abbar} 
     	\Phi_{\alpha\taubar}),\\
     	F_2(\Phi_{\ijbar})&=-\log(b_{\abbar}g^{\abbar}) \label{concave_Function_2},
     \end{align}
 then we can show $F_1$ and $F_2$ are both concave.
    These computations are in Appendix \ref{app:Concavity_lemma}.
    
        Suppose $\Phi$ is a $C^4$ solution to Problem \ref{prob:NewPerturbation} and 
        $X$ is a constant vector field in $\MR\times V$. We  apply $\partial_X$ to equation 
        (\ref{eq:new_Perturbation_in_Inverse_Form_Section1}) and get
        \begin{align}
        	\Phi_{X\tau\taubar}-\Phi_{X\betabar\tau}g^{\alpha\betabar}\Phi_{\alpha\taubar}
        	+\Phi_{\tau\betabar}g^{\alpha\mubar}\Phi_{X\rho\mubar}g^{\rho\betabar}\Phi_{\alpha\taubar}-\Phi_{\tau\betabar}
        	g^{\alpha\betabar}\Phi_{ X\alpha\taubar}=-\epsilon b_{\alpha\betabar 
        	}g^{\alpha\mubar}\Phi_{X\rho\mubar}g^{\rho\betabar}.
        	\label{eq:D_X_NewPerturbation}
        \end{align}
With the linearized operator $L^{\ijbar}\partial_{i\jbar}$, 
(\ref{eq:D_X_NewPerturbation}) is simplified to
    \begin{align}
    	L^{\ijbar}\partial_{i\jbar}(\Phi_X)=0.   \label{eq:DX_NewPerturbation_with_Lij} 
    \end{align}
Then apply $\partial_X$ to equation (\ref{eq:D_X_NewPerturbation}), we get
 \begin{align}
	L^{\ijbar}\partial_{i\jbar}(\Phi_{XX})\geq 0.   \label{eq:DXX_NewPerturbation_with_Lij} 
\end{align}
This is because of the concavity of (\ref{concave_Function}). We  can also  get
 (\ref{eq:DXX_NewPerturbation_with_Lij}) directly by replacing 
$\partial_{\theta}$ and 
$\partial_\gammabar$ in (\ref{eq:20221030-1})-(\ref{eq:20221030-4}) by 
$\partial_X$.
            \subsection{Affine-Manifold-Directional $C^2$ Estimates}
            \label{sec:Torus_Partial_C2Estimates}
            Suppose $X$ is a constant real vector field in $\MR\times V$, parallel to $V$. 
            This is  to say,
            if we denote the projection from $\MR\times V$ to $\MR$ by $\pi_\MR$, then $(\pi_{\MR})_{\ast}(X)=0$.
           By equation (\ref{eq:DXX_NewPerturbation_with_Lij}), we know 
           \begin{align}
           	\Phi_{XX}\leq \max_{\partial \MR\times V} \Phi_{XX},\ \ \ \ \ \text{ in } \MR\times V.
           \end{align}
           Because $\omega_0$ has a lower bound and $\iu\ddbar F(\tau, 
           \ast)$ has a uniform upper bound,
     we can find a constant $C>0$, so 
           that
            \begin{align}
            \max_{\partial \MR\times V} \Phi_{XX}= \max_{\partial \MR\times V} 
            F_{XX}\leq C\omega_0(X, JX).
            \end{align}
        Therefore, 
        \begin{align}
        	\iu \ddbar\Phi(X, JX)=\frac{\Phi_{XX}+\Phi_{JX JX}}{2}\leq C \omega_0(X, JX).
        \end{align}
            This implies 
            \begin{align}
            	\omega_0+\iu\ddbar\Phi(\tau,\ast)\leq (1+C) \omega_0,\ \ \ \ \ \ \text{for 
            	any } \tau\in \partial\MR             
            	\end{align}
        and equivalently
        \begin{align}
        	(g_{\abbar})\leq (1+C) b_{\abbar}.
        \end{align}
    As a consequence, we have
    \begin{align}
    	b_{\abbar} g^{\abbar}\cdot \det (g_{\abbar})\leq n(1+C)^{n-1} \cdot 
    	\det(b_{\abbar}) .   \label{eq:right_hand_side_New_Perturbation_Upper_Bound}
    \end{align}
            
               \subsection{$C^0$ and $C^1$ Estimates}
               \label{sec:C0C1}
               To do the $C^0$ and boundary $C^1$ estimates, we construct $\Psi$ and 
               $\Phi^0$, so that
               \begin{align}
               	\Psi\leq \Phi\leq \Phi^0, \ \ \ \ \ \text{ in } \MR\times V,
               \end{align}
           and
           \begin{align}
           	\Psi=\Phi=\Phi^0, \ \ \ \ \ \text{ on } \partial\MR\times V.
           \end{align}
       
       We let $\Phi^0$ be the solution to Problem \ref{prob:HCMAproductSpace}, with 
       the boundary condition 
       \begin{align}
       	\Phi^0=\Phi,\ \ \ \ \ \ \text{ on } {\partial \MR\times V}.
       \end{align} We easily know that
       \begin{align}
       \Phi\leq \Phi^0, \ \ \ \ \ \text{ in }\MR\times V,
       \end{align} 
   because $\Phi_0$ is a maximal $\Omega_0$-PSH function.
   
   For the construction of $\Psi$, we need to use the estimate 
   (\ref{eq:right_hand_side_New_Perturbation_Upper_Bound}). Locally, we have
   \begin{align}
   	\det(h_{\ijbar})=\epsilon b_{\abbar} g^{\abbar}\cdot \det (g_{\abbar})\leq \epsilon
   	n(1+C)^{n-1} \cdot 
   	\det(b_{\abbar}). \label{418} 
   \end{align}
So, for a solution $\Psi$ to Problem \ref{prob:epsilonCMAproductSpace}, with 
$\varepsilon=\epsilon
n(1+C)^{n-1}$, we have
\begin{align}
	(\OmegaPhi)^{n+1}\leq (\Omega_0+\iu\ddbar\Psi)^{n+1}.
\end{align}
Thus $\Psi\leq \Phi$, given the boundary condition $\Psi=\Phi$ on ${\partial \MR\times 
V}$. The global $C^0$ and $C^1$ estimates for $\Psi$ are known by \cite{Jianchun} 
and \cite{Blocki}. Therefore, we have the global $C^0$ estimate and boundary $C^1$ estimate 
for $\Phi$. 

The $C^{1}$ interior estimate can be derived from boundary estimates with equation 
(\ref{eq:DX_NewPerturbation_with_Lij}).
                  \subsection{$C^2$ Estimates}
                  \label{sec:C2}
                  
                 
                 We first prove the boundary 
                  $C^2$ estimate, then we use equation 
                  (\ref{eq:DXX_NewPerturbation_with_Lij}) to derive the interior estimate.
                  
                  To do the boundary estimate, we need to flatten the boundary. Around 
                  $\tau_0\in\partial \MR$, find a holomorphic map
                  \begin{align}
                  	f: B_{\delta'}(\tau_0)\cap \overline{\MR}\rightarrow \EC=\{\zeta=\xi+\iu \eta\},
                  \end{align}
              for a small $\delta'$. We want that $f'\neq 0$,  $f(\tau_0)=0,$
              \begin{align}
              	f(\partial\MR) \subset \left\{\zeta|\Imm(\zeta)=0 \right\}
              \end{align}
          and
          \begin{align}
          	f(B_{\delta'}(\tau_0)\cap \overline{\MR})\supset 
          	B_\delta^+(0)=\{|\Ree(\zeta)|<\delta, 0\leq \Imm(\zeta)<\delta\},
          \end{align}for a small $\delta$.
          For a point $p_0\in V$, let $\{z^\alpha\}$ be a set of coordinates in 
          $ B_r(p_0)\subset V$, for a small $r$. Without loss of  generality, we can assume that $p_0=0$ in   this coordinate chart. We also require that 
           the coordinate $z^\alpha$ is properly chosen 
          so that  the   natural 
          metric on $ B_r(0)$, as a subset of $\EC^n$, is the metric 
          $\omega_0$. In the following, we will work in the  coordinate chart $ B^+_{\delta}(0)\times B_r(0)$ and estimate second order derivatives at 
          	$(0,0)$. For convenience, we denote $B^+_{\delta}(0)\times B_r(0)$ by $\mathcal{D}$ and denote $\{\Imm(\zeta)=0\}\times B_r(0)$ by $\Gamma$.
          
          In the coordinate chart $ \mathcal{D}$, 
           $\Phi(\zeta,  \vec  z)$ 
          satisfies
          \begin{align}\label{eq:new_Perturbation_after_coordinate_Trans_At_Boundary}
          	\Phi_{ \zeta\zetabar}-\Phi_{\zeta\betabar}g^{\abbar} \Phi_{\alpha\zetabar}=
          	\epsilon\cdot k(\zeta)\cdot b_{\abbar} g^{\abbar}.
          \end{align}
           Here $k=\frac{1}{|f'|^2}$, so it is a positive and smooth function  $\zeta$.  
      \begin{align}
      	\frac{1}{K}<k< K,
      \end{align}for a constant $K$.  
  
  The second order derivatives of $\Phi$ at $(0,0)$ in $\Gamma$ directions are known, they depend on the boundary value $F$.  
      Let $X$ be a constant vector field parallel to $\Gamma$ and with $|X|=1$. We need to estimate $\Phi_{X\eta}$, then using equation we can control $\Phi_{\eta\eta}$.   
      
      The method of estimating $\Phi_{X\eta}$ is similar to the method used in   \cite{GuanComplex}. According to the estimate of section \ref{sec:C0C1}, there is a constant $C_1$, so that 
      \begin{align}
      	|\Phi_X|\leq C_1.
      \end{align}
  In the following, we will show $\Phi_{X\eta}(0,0)\leq C_2$, for a constant $C_2$. 
  
  First, we need to derive an equation satisfied by $\Phi_{X}$. Applying $\partial_X$ to equation (\ref{eq:new_Perturbation_after_coordinate_Trans_At_Boundary}) gives
   \begin{align}
 	\Phi_{X\zeta\zetabar}-\Phi_{X\betabar\zeta}g^{\alpha\betabar}\Phi_{\alpha\zetabar}
 	+\Phi_{\zeta\betabar}g^{\alpha\mubar}\Phi_{X\rho\mubar}g^{\rho\betabar}\Phi_{\alpha\zetabar}-&\Phi_{\zeta\betabar}
 	g^{\alpha\betabar}\Phi_{ X\alpha\zetabar}\\
 	&=-\epsilon  k b_{\alpha\betabar 
 	}g^{\alpha\mubar}\Phi_{X\rho\mubar}g^{\rho\betabar}+\epsilon(\partial_Xk) b_{\alpha\betabar 
 }g^{\alpha\betabar}.
 	\label{eq:D_X_NewPerturbation_after_Coordinate_Change_boundary}
 \end{align}
We introduce the following operator $\ML$, which is a scalar function multiple of $L$, after coordinate transformation. Equivalently, $\ML$ can also be considered as the linearization operator of (\ref{eq:new_Perturbation_after_coordinate_Trans_At_Boundary}):
\begin{align}
	\ML=\ML^{\ijbar}\partial_{\ijbar},
\end{align}                         
with
\begin{align}
	\left(\ML^{\ijbar}\right)=\left(
	\begin{array}{cc}
		\ML^{0\overline 0}&\ML^{0\betabar}\\
		\ML^{\alphabar 0}&\ML^{\alpha\betabar}
	\end{array}\right)
=
\left(
\begin{array}{cc}
	1&-\Phi_{\rho\zetabar} g^{\rho\betabar}\\
-\Phi_{\zeta\mubar} g^{\alpha\mubar}&\epsilon k b_{\mu\rhobar} g^{\mu\betabar} g^{\alpha\rhobar}+\Phi_{\zeta\mubar} g^{\alpha\mubar}\Phi_{\rho\zetabar} g^{\rho\betabar}
\end{array}\right).
\end{align}Here $i,j$ run from $0$ to $n$, and the $0$-th coordinate is $\zeta$.
With this operator, equation (\ref{eq:D_X_NewPerturbation_after_Coordinate_Change_boundary}) becomes
\begin{align}
	\ML^{i\overline{j}}\partial_{i\overline{j}}(\Phi_{X})=\epsilon (\partial_Xk) b_{{\alpha\betabar}} g^{{\alpha\betabar}}.
	\label{eq:DX_newPerturbation_after_coordinate_Trans_at_Boundary_with_L}
\end{align}
Since we have the metric lower bound estimate, Proposition \ref{prop:metric_lower_bound}, we know the right-hand side of (\ref{eq:DX_newPerturbation_after_coordinate_Trans_at_Boundary_with_L}) is bounded.  We can assume, for a constant $\tilde{C}$, 
\begin{align}
	-\tilde{C} \leq \ML^{i\overline{j}}\partial_{i\overline{j}}(\Phi_{X}).
\end{align}

Then we construct a barrier function $u$, so that $u\geq \Phi_{X}$ and $u(0,0)=\Phi_X(0,0)$. Thus we can get an upper bound for $\Phi_{X\eta}$. The barrier function is
\begin{align}
	u=l+C_3\left(|z^\alpha|^2+\xi^2\right)+C_4\eta-C_5\eta^2+C_6(\Phi-\Psi).
\end{align} In above, $l$ is the $\Gamma$-directional linearization of $\Phi_{X}$ at $(0,0)$. That's to say, 
\begin{align}
	l(0,0)&=\Phi_{X}(0,0),\\
	\partial_\eta l&=0
\end{align}
and 
\begin{align}
	\nabla_\Gamma l(0,0)=\nabla_\Gamma\Phi_{X}(0,0).
\end{align}
$\Psi$ is a solution to Problem \ref{prob:epsilonCMAproductSpace} with $\varepsilon=\epsilon n (1+C)^{n-1}+1$ and 
\begin{align}
	\Psi=\Phi,\ \ \ \ \ \ \text{ on } {\partial \MR\times V}.
\end{align} The $\Psi$ constructed here is even smaller than the $\Psi$ constructed in section (\ref{sec:C0C1}), so $\Psi\leq \Phi$. And according to the $C^2$ estimate for $\Psi$ \cite{Blocki}, we have
\begin{align}
	\left(
	\begin{array}{cc}
		\Psi_{\zeta\zetabar}& \Psi_{\zeta\betabar}\\
		\Psi_{\alpha\zetabar}& \Psi_{{\alpha\betabar}}+b_{\alpha\betabar}
	\end{array}
	\right)>\frac{1}{C_7}	\left(
	\begin{array}{cc}
	1&         \\
	              &b_{\alpha\betabar}
	\end{array}
	\right).\label{eq:psi_lowerbound}
\end{align} 

In the following, we show that we can properly choose parameters $C_3, C_4, C_5, C_6$, so that 
\begin{align}
	\ML^{i\overline{j}}\partial_{i\overline{j}}u\leq- \tilde{C},\ \ \ \ \ \text{ in }\Omega,
\end{align} 
and 
\begin{align}
	u\geq \Phi_{X},\ \ \ \ \ \ \text{ on }\partial\Omega.
\end{align} 
Then the comparison principle implies that $u\geq \Phi_X$ in $\Omega$.

We compute and estimate $\ML^{i\overline{j}}\partial_{i\overline{j}}u$ term by term:

(i) $\ML^{i\overline{j}}\partial_{i\overline{j}} (l+C_4 \eta)=0$, because $l+C_4\eta$ is a linear function.

(ii)Using the metric lower bound estimate, Proposition \ref{prop:metric_lower_bound}, we have \begin{align}
	\ML^{i\overline{j}}\partial_{i\overline{j}}(|z^\alpha|^2+\xi^2)=\epsilon k b_{\theta\gammabar} g^{\theta\alphabar} g^{\alpha\gammabar}
	  +g^{\theta\alphabar}\Phi_{\theta\zetabar}g^{\alpha\gammabar} \Phi_{\zeta\gammabar} +\frac{1}{2}\leq C_8 \epsilon K+g^{\theta\alphabar}\Phi_{\theta\zetabar}g^{\alpha\gammabar} \Phi_{\zeta\gammabar}.  \label{20221109_440}
\end{align}
Here $C_8$ depends on the  metric lower bound.

(iii) $\ML^{i\overline{j}}\partial_{i\overline{j}}(-\eta^2)=-\frac{1}{2}.$

(iv) For $\ML^{i\overline{j}}\partial_{i\overline{j}}(\Phi-\Psi)$, we split it into two terms:
 $	\ML^{i\overline{j}}\partial_{i\overline{j}}(\Phi)
 +\ML^{{\alpha\betabar}}b_{{\alpha\betabar}}$ and 
 $-(	\ML^{i\overline{j}}\partial_{i\overline{j}}(\Psi)
 +\ML^{{\alpha\betabar}}b_{{\alpha\betabar}})$.
Using $b_{\alpha\betabar}+\Phi_{\alpha\betabar}=g_{\alpha\betabar}$, we get
\begin{align}
	\ML^{i\overline{j}}\partial_{i\overline{j}}(\Phi)
	+\ML^{{\alpha\betabar}}b_{{\alpha\betabar}}=2\epsilon k b_{\alpha\betabar} g^{\alpha\betabar}. \label{20221112441}
\end{align}
For the right-hand side of (\ref{20221112441}), we use the metric lower bound estimate and get
\begin{align}
	2\epsilon k b_{\alpha\betabar} g^{\alpha\betabar}\leq 2 \epsilon K C_9.
\end{align}
For $\ML^{i\overline{j}}\partial_{i\overline{j}}(\Psi)
+\ML^{{\alpha\betabar}}b_{{\alpha\betabar}}$, we  use (\ref{eq:psi_lowerbound}) and get
\begin{align}
	\ML^{i\overline{j}}\partial_{i\overline{j}}(\Psi)
	+\ML^{{\alpha\betabar}}b_{{\alpha\betabar}}\geq \frac{1}{C_7}
	\left(g^{\alpha\etabar}\Phi_{\zeta\etabar} g^{\mu\betabar} \Phi_{\mu\zetabar}b_{\alpha\betabar}\right).
\end{align}
Note that we have already made the assumption, when choosing the coordinate chart, that, in $\mathcal{D}$, $b_{{\alpha\betabar}}=\delta_{\alpha\betabar}$. So
\begin{align}
	\ML^{i\overline{j}}\partial_{i\overline{j}}(\Psi)
	+\ML^{{\alpha\betabar}}b_{{\alpha\betabar}}\geq \frac{1}{C_7}
	\left(g^{\alpha\etabar}\Phi_{\zeta\etabar} g^{\mu\alphabar} \Phi_{\mu\zetabar}\right).
	\label{20221109_444}
\end{align}
The right-hand side of (\ref{20221109_444}) can be used to control the right-hand side of (\ref{20221109_440}).

In sum, we have
\begin{align}
	\ML^{i\overline{j}}\partial_{i\overline{j}} u
	\leq (C_3\cdot \epsilon K C_8+C_6\cdot 2\epsilon K C_9-\frac{C_5}{2})+
	   g^{\alpha\etabar}\Phi_{\zeta\etabar} g^{\mu\alphabar} \Phi_{\mu\zetabar}
	   (C_3-\frac{C_6}{C_7}).
\end{align}
We need to choose $C_3, C_5, C_6$ so that 
\begin{align}
	C_6\geq C_3\cdot C_7\label{conditionC6}
\end{align}
and
\begin{align}
	C_5\geq 2(C_3\cdot 2 \epsilon K C_8+C_6\cdot 2\epsilon K C_9+\tilde{C}). \label{conditionC5}
\end{align}

We also want $u\geq \Phi_{X}$ on $\partial\Omega$. We need to choose $C_3$ big enough, so that
\begin{align}
	C_3(\sum_\alpha|z^\alpha|^2+\xi^2)+l\geq \Phi_X+(\sum_\alpha|z^\alpha|^2+\xi^2), \ \ \ \ \text{ on } \Gamma.
	\label{20221109448}
\end{align} 
This requires 
\begin{align}
	C_3 \geq  \max_{\Gamma}|D^2(\Phi_{X}|_\Gamma)|+1. \label{conditionC3}
\end{align}
We note that 
\begin{align}
	C_4\eta-C_5\eta^2+C_6(\Phi-\Psi)=0, \ \ \ \ \ \ \text{ on }\Gamma,
\end{align}
so $u\geq \Phi_{X}$ on $\Gamma$, given (\ref{20221109448}) is valid.
To make $u\geq \Phi_{X}$ on $\partial \Omega-\Gamma$, we choose $C_4$ big enough. Given (\ref{20221109448}), we have 
\begin{align}
	u\geq \Phi_X+\delta_1, \ \ \ \ \ \ \text{ on }\partial\Gamma,
\end{align}
for a small positive constant $\delta_1$. Then for a small $\delta_2\in(0,r)$,  
\begin{align}
	l+C_3\left(|z^\alpha|^2+\xi^2\right)-C_5\eta^2+C_6(\Phi-\Psi)\geq \Phi_X,
	\ \ \ \ \ \  \text{ on }\{\eta\leq\delta_2\}\cap \partial \Omega.
\end{align}
$\delta_2$ depends on $\delta_1, C_5, C_6$,  second order derivatives of $F$ and the norms of gradients of $\Psi, \Phi$. We also have
 \begin{align}
	l+C_3\left(|z^\alpha|^2+\xi^2\right)-C_5\eta^2+C_6(\Phi-\Psi)>\Phi_{X}-C_{10},\ \ \ \ \ \ \text{ in }\Omega,
\end{align}
for a constant $C_{10}$,  depending on $C_5, C_6$, second order derivatives of $F$, the $C^1$ norm of $\Phi$ and the $C^0$ norm of $\Psi$. We can choose
\begin{align}
	C_4>\frac{C_{10}}{\delta_2}. \label{conditionC4}
\end{align}
Then $u\geq \Phi_X$ on $\partial \MD$.

In sum, we choose $C_3$ large enough with condition (\ref{conditionC3}), then choose $C_6$ with condition (\ref{conditionC6}), then choose $C_5$ with condition (\ref{conditionC5}) and finally choose $C_4$ according to condition (\ref{conditionC4}).

Thus, we have an upper bound for $\Phi_{X\eta}$. To get the lower bound, we simply replace $\partial_X\Phi$ by $\partial_{-X}\Phi$.

Then we can use equation (\ref{eq:new_Perturbation_after_coordinate_Trans_At_Boundary}) to get the estimate of $\Phi_{\eta\eta}$. Just note that $4\Phi_{\zeta\zetabar}=\Phi_{\xi\xi}+\Phi_{\eta\eta}$. So
\begin{align}
	\Phi_{\eta\eta}=-\Phi_{\xi\xi}+4\Phi_{\zeta\betabar}g^{\abbar} \Phi_{\alpha\zetabar}+4
	\epsilon\cdot k(\zeta)\cdot b_{\abbar} g^{\abbar}.
	\label{20221109455}
\end{align}
The estimate of the right-hand side of (\ref{20221109455}) depends on the boundary value $F$, the metric lower bound estimate and the estimate of $\Phi_{X\eta}$.

                  Given the boundary estimate, we can go back to the original $(\tau, \vec z)$ coordinates and use equation (\ref{eq:DXX_NewPerturbation_with_Lij}). Similar to section \ref{sec:Torus_Partial_C2Estimates}, for any constant vector $X$ in $\MR\times V$, we have 
                  \begin{align}
                  	\Phi_{XX}\leq \max_{\partial \MR\times V} \Phi_{XX}, \ \ \ \ \ \ \text{ in } \MR\times V.
                  \end{align} 
              For the lower bound estimate of $\Phi_{XX}$, we have
                   \begin{align}
                   	-\Phi_{XX}&=\Phi_{JX JX}-2\iu\ddbar\Phi(X, JX)\\
                   	&=\Phi_{JX JX}-2(\Omega_0+\iu\ddbar\Phi)(X, JX)+2\Omega_0(X, JX).
                   \end{align}
                   Then, using $\OmegaPhi\geq0$, we get the lower bound of $\Phi_{XX}$.
                   
                   In sum, we get
                   \begin{align}
                   	|\Phi|_{C^2(\overline{\MR}\times V)}\leq C_{11},
                   \end{align}
                   for a constant $C_{11}$, which depends on $\epsilon$, $|F|_{C^3(\overline{\MR}\times V)}$, $\omega_0$, the metric lower bound estimate and the boundary of $\MR$. In particular,  when $\epsilon\rightarrow 0$, the constant $C_{11}$ does not go to $\infty$. However, we don't need to use this fact.
   \subsection{$C^{2,\alpha}$ Estimates}
       \label{sec:C2alpha}
                     With the $C^2$ estimate in the previous section, we know the operator $\ML$ and $L$ are uniform elliptic.  We only need to prove the boundary $C^{2,\alpha}$ estimate. Then, with the uniform ellipticity, concavity, $C^2$ estimate and the boundary $C^{2,\alpha}$ estimate, we can derive the interior $C^{2, \alpha}$ estimate with standard method \cite{QingHanBook}.
                     
                     For the boundary $C^{2,\alpha}$ estimate we need to flatten the boundary again. Adopting notation of section \ref{sec:C2}, we know that $\Phi$ satisfies the equation
                     \begin{align}
                     	\ML^{i\overline{j}}\partial_{i\overline{j}}(\Phi_{X})=\epsilon (\partial_Xk) b_{{\alpha\betabar}} g^{{\alpha\betabar}}
                     	\ \ \ \ \ \ \text{  in $\MD$}.
                     	\label{eq:DX_newPerturbation_after_coordinate_Trans_at_Boundary_with_L_secondtime}
                     \end{align}
              We construct a function
                 \begin{align}
                 	\MF(\zeta, \vec z)=\partial_X F(\Ree(\zeta), \vec z).
                 \end{align}
             Then
             \begin{align}
             	\Phi_X-F&=0,&\text{ on }\Gamma;\\
             	\ML^{i\overline{j}}\partial_{i\overline{j}}(\Phi_X-F)&=\epsilon (\partial_Xk) b_{{\alpha\betabar}} g^{{\alpha\betabar}}-\ML^{i\overline{j}}\partial_{i\overline{j}} \MF,  &\text{ in }\MD.\label{1109459}
             \end{align}
         The right-hand side of (\ref{1109459}) is bounded, according to the $C^2$ estimate in the previous section, so
              we can use Theorem  1.2.16 of \cite{QingHanBook} and get the $C^\alpha$ estimate for $\partial_\eta(\Phi_X-\MF)=\Phi_{X\eta}$ in a small neighborhood of $0$ in $\Gamma$, for an $\alpha\in(0,1)$. Using equation (\ref{20221109455}), we get the $C^\alpha$ estimate for $\Phi_{\eta\eta}$. 
                 
                     
 \subsection{Existence of Smooth Solutions by the Method of Continuity}
 \label{sec:existence_by_method_continuity}
 Suppose $F\in C^{\infty}({\partial \MR\times V})$ satisfies condition (\ref{20221107BoundarySOmega_0Convexity}), for a constant section $F$ of $T_{2,0}^\ast(V)$. Then according to Lemma \ref{lemma:monotone}, 
 \begin{align}
 	\sigma F(\tau, \ast) \text{ is strictly \SO-convex for any $\tau\in \MR$ and any $\sigma\in[0,1]$ }.
 \end{align}
                        Consider the set 
                        \begin{align}
                        \mathscr{S}=	\{\sigma\in[0,1]\big|& \text{Problem 
                        \ref{prob:NewPerturbation} with 
                        	boundary value $s\cdot F$ has a solution $\Phi^s$,  for any 
                        	$s\leq 
                        	\sigma$, } \\
                        &\ \ \ \text{and $\Phi^s$ is a continuous curve in $C^4(\overline{\MR}\times V)$ with $C^2$ topology}\} .
                        \end{align}
                    Obviously, $\mathscr{S}$ is non-empty, since it contains $0$. So, if $\mathscr{S}$ is both open and closed, then $\mathscr{S}=[0,1]$
                    
                    Before we prove the openness and closeness, we point out that if a solution $\Phi$ is in $C^{2,\alpha}(\overline{\MR}\times V)$ then we can use the standard bootstrap technique (Theorem 5.1.9 and 5.1.10 of \cite{QingHanBook})) to show that $\Phi$ is actually in $C^{\infty}(\overline{\MR}\times V)$. This is because of the condition (\ref{eq:Problem_NewPerturbation_NonDegenerateCondition}) and the ellipticity of the equation (\ref{eq:new_Perturbation_in_Form_Form_Section1}). 
                        
                        The openness can be proved with standard implicit function theorem.  This is because of the condition (\ref{eq:Problem_NewPerturbation_NonDegenerateCondition}) and the ellipticity. Without such condition, the openness can be quite difficult to prove.  For example, in \cite{CFH}, we used Nash-Moser inverse function theorem to prove an openness result for geodesic equations.
                        
                        For the closeness, given  $\{\sigma_i\}_{i\in\EZ^+}\subset \mathscr{S}$ with
                        $\lim_{i\rightarrow\infty}\sigma_i=\sigma_{\infty}$, we need to show that $\sigma_\infty\in \mathscr{S}$. 
                          According to the $C^{2,\alpha}$ estimate, the sequence of solutions $\Phi^{\sigma_i}$ satisfy
                          \begin{align}
                                 |\Phi^{\sigma_i}|_{C^{2,\alpha}}\leq C
                          \end{align}
                      and 
                      \begin{align}
                      	\Omega_0+\iu\ddbar \Phi^{\sigma_i}\geq \frac{1}{C} (\Omega_0+\sqrt{-1} d\tau\wedge \overline{d\tau}),
                      \end{align}
                  for a constant $C$, which depends on $\epsilon$. It's easy to know that $\Phi^{\sigma_i}$ is a Cauchy sequence in $C^0(\overline{\MR}\times V)$. So, using interpolation, we can find $ \Phi^{\sigma_\infty} \in C^{2,\frac{\alpha}{2}}(\overline{\MR}\times V)$ and
                  \begin{align}
                  	\Phi^{\sigma_i}\rightarrow \Phi^{\sigma_\infty}, \ \ \ \ \ \ \ \   \text{ in $C^{2,\frac{\alpha}{2}}$ norm.}
                  \end{align}
              The $C^{2,\frac{\alpha}{2}}$ convergence  implies $\Phi^{\sigma_\infty}$ satisfies equation (\ref{eq:new_Perturbation_in_Form_Form_Section1}) and  condition (\ref{eq:Problem_NewPerturbation_NonDegenerateCondition}). Therefore, $\Phi^{\sigma_\infty}$ is a solution to Problem \ref{prob:NewPerturbation}, with boundary value ${\sigma_\infty}  F$.
              
              It remains to show that, for any $\lambda\in [0,\sigma_{\infty})$, there is a solution $\Phi^{\lambda}$ and $\{\Phi^{\lambda}|\lambda\in[0,\sigma_{\infty}]\}$ is a continuous curve with $C^2$ topology.
              
              In the following, when we say solution $\Phi^\theta$, we always mean a solution to Problem \ref{prob:NewPerturbation} with boundary value $\theta F$.
              For any $\nu<\sigma_{\infty}$, let $\delta_\nu=\frac{\sigma_{\infty}-\nu}{2}$.   There is a $\sigma_k>\nu+\delta$ because $\sigma_k\rightarrow \sigma_\infty$. $\sigma_k$ being in $ \mathscr{S}$ implies solution $\Phi^{\nu}$ with boundary value $\nu  F$ exists and 
  $
              	\{\Phi^s| s\in[0, \nu+\delta_\nu]\}
$ is a continuous curve with $C^2$ topology. So
\begin{align}
	\{\Phi^{\lambda}|\lambda\in[0,\sigma_{\infty})\}
	=\bigcup_{\nu<\sigma_\infty}\{\Phi^{\lambda}|\lambda\in[0,\nu+\delta_\nu)\}
\end{align} is $C^2$ continuous everywhere.  Therefore,  for any sequence $\nu_k\rightarrow \sigma_\infty$, solution $\Phi^{\nu_k}$ has uniform $C^{2,\alpha}$ estimate. It's easy to know $\Phi^{\nu_k}$ converges to $\Phi^{\sigma_\infty}$ in $C^0$ norm, so, by interpolation, we know $\Phi^{\nu_k}$ converges to $\Phi^{\sigma_\infty}$ in $C^2$ norm and the curve   $\{\Phi^{\lambda}|\lambda\in[0,\sigma_{\infty}]\}$ is $C^2$ continuous everywhere.  

In sum, we have the following.
\begin{theorem}[Existence of Smooth Solutions to Problem \ref{prob:NewPerturbation} and Convexity Estimates]
	\label{existence}
Suppose that, for a constant $\delta>0$ and a constant section $S$ of $T_{2,0}^{\ast}(V)$, $F\in C^{\infty}({\partial \MR\times V})$ satisfies
\begin{align}
	F(\tau, \ast) \text{ is \SO-convex of modulus $>\delta$,  for any $\tau\in\partial \MR$}.
\end{align}
Then  Problem \ref{prob:NewPerturbation} with boundary value $F$ has a unique and smooth solution $\Phi$. In addition, 
\begin{align}
	\Phi(\tau, \ast) \text{ is \SO-convex of modulus $>\delta$,  for any $\tau\in\MR$},
\end{align}
and, consequently, 
	\begin{align}
	\omega_0+\iu\ddbar	\Phi(\tau, \ast) > \delta\omega_0 \text{, \ \ \ \ \ \  for any $\tau\in \MR$}. 
\end{align}
\end{theorem}
\section{Estimates For Homogenous Monge-Amp\`ere Equations}
            \label{sec:estimates_No_Assumption}
            In this section, we prove estimates for solutions to Problem \ref{prob:HCMAproductSpace} and Problem \ref{prob:geodesicMongeAmpereonStrip}. Solutions we talk about in this section all have the same boundary value $F$ which satisfies condition (\ref{condition:convex_module_theorem}).
            
            Suppose $\Phi^{\epsilon}$ is the solution to Problem \ref{prob:NewPerturbation} and $\Phi^0$ is the solution to Problem \ref{prob:HCMAproductSpace}. We will show that $\Phi^{\epsilon}$ converges to $\Phi^0$ in $C^0$ norm and $\Phi^{0}$ satisfies estimates, which are satisfied by $\Phi^\epsilon$.
            
            Let $\Psi^{\epsilon}$ be the solution to Problem \ref{prob:epsilonCMAproductSpace} with $\varepsilon=\epsilon n (1+C)^{n-1}$, where the constant $C$ is from (\ref{418}). Then we know 
            \begin{align}
            	\Psi^{\epsilon}\leq \Phi^{\epsilon} \leq \Phi^{0}, \ \ \ \ \ \ \text{ in }{ \MR\times V}.
            \end{align}
        According to estimates, in \cite{Chen2000} or \cite{Blocki},  for solutions to Problem \ref{prob:epsilonCMAproductSpace},  $\Psi^{\epsilon}\rightarrow \Phi^{0}$ in $C^0$ norm, so $\Phi^{\epsilon} \rightarrow \Phi^{0}$ in $C^0$ norm. In particular, for any $\tau\in \MR$, 
        \begin{align}
        	\Phi^{\epsilon}(\tau, \ast)\rightarrow \Phi^{0}(\tau, \ast), \ \ \ \ \ \ \text{ in } C^0 \text{ norm. }
        \end{align}
    
   According to Theorem \ref{existence},  in every local coordinate chart, 
    \begin{align}
    	\Phi^{\epsilon}(\tau, \ast)+(1-\mu) b_{\alpha\betabar}z^\alpha \zbetabar+\Ree\left(S_{\abbar} z^\alpha\zbetabar\right)
    	\label{11_53}
    \end{align}
is a convex function, for any $\tau\in\MR$. Then, because of the $C^0$ convergence of $\Phi^\epsilon\rightarrow \Phi^0$, we can replace $\Phi^{\epsilon}$ by $\Phi^{0}$ in (\ref{11_53}) and get
 \begin{align}
	\Phi^{0}(\tau,\ast)+(1-\mu) b_{\alpha\betabar}z^\alpha \zbetabar+\Ree\left(S_{\abbar} z^\alpha\zbetabar\right)
\end{align}
is a convex function, for any $\tau\in\MR$. Thus $\Phi^{0}(\tau, \ast)$ is  \SO-convex of modulus $\geq\mu$, for any $\tau\in\MR$. 

For the metric lower bound estimate, the proof is standard, we only need to do an integration by parts. Theorem \ref{existence} implies that, for any positive function $\eta$,
\begin{align}
	\int_{V} (\omega_0(1-\mu)+\iu\ddbar \Phi^{\epsilon}(\tau, \ast))\wedge \omega_0^{n-1} \eta \geq 0,
\end{align}
for any $\tau\in\MR$.
Then for any $\eta$ with sufficiently small support, we can find $\rho_0$, so that
\begin{align}
	\omega_0=\iu\ddbar\rho_0, \ \ \ \ \ \ \text{ in the support of $\eta$}.
\end{align}
Thus
\begin{align}
	\int_{V} (\rho_0(1-\mu)+ \Phi^{\epsilon}(\tau, \ast))\wedge \omega_0^{n-1}\wedge \iu\ddbar \eta \geq 0,
\end{align}for any $\tau\in\MR$.
Let $\epsilon\rightarrow 0$, we get
\begin{align}
	\int_{V} (\rho_0(1-\mu)+ \Phi^{0}(\tau, \ast))\wedge \omega_0^{n-1}\wedge \iu\ddbar \eta \geq 0,
\end{align}for any $\tau\in\MR$. So
\begin{align}
	\omega_0+\iu\ddbar\Phi^0(\tau, \ast)\geq \mu \omega_0,
\end{align}
 in the weak sense, for any $\tau\in\MR$.
 
 Thus, Theorem \ref{thm:estimate_HCMA_ProductSpace} is proved. 
 
 \begin{remark}
 	Here, we cannot use Lemma \ref{lemma:Equivalence Between Module of Convexity and Degree of Convexity} to directly derive metric lower bound estimate from convexity estimate, because $\Phi^0(\tau)$ may not be $C^2$ continuous and the degree of \SO-convexity may not be well defined.
 \end{remark}

\appendix
\section{Algebra Lemmas}\label{app:algebraLemmas}

\subsection{Lemmas for Convexity Estimate}
In this appendix, we show that if $\varphi$ is $C^2$ then Definition 
\ref{def:SOmegaZeroConvexityC0_1117} and Definition \ref{def:SOmegaConvexityMatrix_1117}
 are equivalent.  Furthermore, the modulus of \SO-convexity and the degree of \SO-convexity also coincide.
 The main results are  Lemma \ref{lemma:equivalent_definition_S_Omega_Convexity}
and
Lemma \ref{lemma:Equivalence Between Module of Convexity and Degree of Convexity}.
 \begin{lemma}[Equivalent Definitions of  Strict \SO-Convexity]
 	\label{lemma:equivalent_definition_S_Omega_Convexity}
 	Suppose that  $\varphi$ is a $C^2$ continuous function on $V$. Then it satisfies the  	condition of  	Definition \ref{def:SOmegaZeroConvexityC0_1117} if and only if it satisfies the condition of 
 	Definition \ref{def:SOmegaConvexityMatrix_1117}
 \end{lemma}
\begin{lemma}
	[Equivalence between Modulus of Convexity and Degree of Convexity]
	\label{lemma:Equivalence Between Module of Convexity and Degree of Convexity}
	Suppose that  $\varphi$ is a $C^2$ continuous function on $V$ and  $S$ is a constant section of $T_{2,0}^{\ast}(V)$. Then $\varphi$ is \SO-convex of degree $>\delta$ if and only if it is \SO-convex of modulus $>\delta$.
\end{lemma}

In the proof of these Lemmas and also in other parts of the paper, we 
need to use the following Autonne-Takagi factorization and its corollary  Lemma \ref{lemma:simultaneous_diagonalization}. The following Autonne-Takagi factorization is the Corollary 
4.4.4(c) 
of \cite{HornMatrix}.

\begin{lemma}
	[Autonne-Takagi Factorization] 
	\label{lemma:Autonne_Takagi_Factorization}Given a complex valued symmetric 
	matrix $S$, there is a unitary matrix $U$ such that 
	\begin{align}
		S=U^T \Sigma U
	\end{align}
in which $\Sigma$ is a non-negative diagonal matrix. And obviously the Hermitian 
matrix $S\overline S$ has a decomposition 
\begin{align}
	S\overline S=U^T \Sigma^2 \overline U.
\end{align}
\end{lemma}

With Autonne-Takagi Factorization, we can prove the following lemma.
\begin{lemma}\label{lemma:simultaneous_diagonalization}
	Suppose $A$ is an $n\times n$ positive definite Hermitian matrix and $B$ is a 
	complex $n\times n$ symmetric matrix. Then we can find an $n\times n$ 
	invertible 
	matrix $P$, so that
	\begin{align}
		PAP^\ast=I,\ \ \ \ \ \ \ \ \ P BP^T=\Lambda,  \label{20221106A13condtion}
	\end{align}
in which $\Lambda$ is a non-negative diagonal matrix.
\end{lemma}
\begin{proof}[Proof of Lemma \ref{lemma:simultaneous_diagonalization}]
   Since $A$ is a positive definite Hermitian matrix, we can find an invertible matrix 
   $Q$ so that 
   \begin{align}
   	QAQ^\ast=I.
   	\end{align}
   Then we apply Autonne-Takagi facorization to $QBQ^T$ and find $R\in U(n)$ so 
   that
   \begin{align}
   	R(QBQ^T)R^T=\Lambda.
   \end{align}
$P=RQ$ satisfies condition  (\ref{20221106A13condtion}).
\end{proof}

The proof of Lemma \ref{lemma:equivalent_definition_S_Omega_Convexity} 
essentially depends on the following lemma.
\begin{lemma}
	\label{lemma:S_Omega_0_convexity_LinearAlgebra}
	Suppose $U, V, W$ are real valued $n\times n$ matrices and $U, V$ are 
	symmetric. 
	Let
	\begin{align}
		A=\frac{1}{4}(U+W)+\frac{\iu}{4}(V-V^T), \label{20221106UVtoA}\\
		B=\frac{1}{4}(U-W)-\frac{\iu}{4}(V+V^T). \label{20221106UVtoB}
	\end{align}
Then 
\begin{align}
	\left(
	\begin{array}{cc}
	    U& V\\ V^T &W
	\end{array}
	\right)>0 \label{eqt:20221106UVW>0}
\end{align}
if and only if 
\begin{align}
	A>0\text{\ \ \ and \ \ \ } B \Abarinverse\  \Bbar \Ainverse<1. 
\end{align}
\end{lemma}
We first prove Lemma \ref{lemma:S_Omega_0_convexity_LinearAlgebra}, then 
Lemma 
\ref{lemma:equivalent_definition_S_Omega_Convexity} follows easily.

\begin{proof}
	[Proof of Lemma \ref{lemma:S_Omega_0_convexity_LinearAlgebra}] 
	
	($\Rightarrow$) Given $U, V, W$ satisfying \ref{eqt:20221106UVW>0}, 
	 we can 
	construct a strictly convex quadratic 
	polynomial on $\EC^n$, with coordinate $z^\alpha=x^\alpha+\iu y^\alpha$,
	\begin{align}
		H({\bf z})= U_{\alpha\beta} x^\alpha x^\beta + 2 V_{\alpha\beta} x^\alpha 
		y^\beta+W_{\alpha\beta}y^\alpha y^\beta. \label{20221106A20}
	\end{align}
$H$ is a strictly PSH function, since it's strictly convex. And it's straightforward to 
check 
that 
\begin{align}
	\partial_{\alpha\betabar} H=A_{\alpha\betabar} \text{ \ \ \ and \ \ \ \ } 
	\partial_{\alpha\beta} H=B_{\alpha\beta}.
\end{align}
Therefore, we know $A>0$.  Then using Lemma \ref{lemma:simultaneous_diagonalization} 
we find matrix $P$, so that 
	\begin{align}
	PAP^\ast=I,\ \ \ \ \ \ \ \ \ P BP^T=\Lambda, 
\end{align}
in which $\Lambda=\diag(\lambda_1, \lambda_2,\ ... \ ,\lambda_n)$ is non-negative. We consider a new set of coordinates
$\{\zeta^\alpha\}$, with 
\begin{align}
	z^\alpha=P^\alpha_\beta \zeta^\beta.  \label{A13_1120}
\end{align}
With this new coordinate
\begin{align}
(	\partial_{\zeta^\alpha \zeta^\beta} H)
=(P_{\alpha}^\mu P_\beta^{\rho} \partial_{z^\mu 
	z^\rho} H)=P B P^T=\Lambda;\\
(	\partial_{\zeta^\alpha \zeta^\betabar} H)
=(P_{\alpha}^\mu \overline{P_\beta^{\rho}} 
	\partial_{z^\mu 
	z^\rhobar} H)=P A P^\ast=I.
\end{align} In above, $P=(P_\alpha^\beta)$, where $\alpha$ is the row index and $\beta$ is the column index. It's obvious that the  linear change of coordinates (\ref{A13_1120}) does not affect the fact that $H$ is a convex function. So the restriction of $H$ to a complex line 
\begin{align}
	L_\alpha=\{\zeta^\alpha=\tau, \zeta^\mu=0, \text{ for } \mu\neq \alpha |\tau\in\EC\}
\end{align}
is still a convex function. Therefore,
\begin{align}
	H|_{L_\alpha}=\tau\taubar+\lambda_\alpha\frac{\tau^2+\taubar^2}{2}
\end{align} is a convex function of $\tau$.
This implies $|\lambda_\alpha|<1$. So $\Lambda^2<1$ and 
\begin{align}
	B \Abarinverse \ \Bbar \Ainverse=P^{-1} \Lambda^2 P<1.\label{A18_1120}
\end{align}
($\Leftarrow$) For another direction, we use Lemma 
\ref{lemma:simultaneous_diagonalization} to diagonalize $A, B$ simultaneously. Since $A>0$, we can find 
$R$ so that 
\begin{align}
	B=R\Lambda R^T, \ \ \ \ \text{and } \ \ \ \ \ A=RR^\ast.
\end{align}
Let $R=R_1+\iu R_2.$ Then
\begin{align}
	&B=R_1\Lambda R_1^T-R_2 \Lambda R_2^T+\iu (R_2\Lambda R_1^T+R_1\Lambda 
	R_2^T); \label{20221106A30}\\
	&A=R_1 R_1^T+R_2  R_2^T+\iu (R_2 R_1^T-   R_1 	R_2^T).\label{20221106A31}
\end{align}
We can get $U, V,  W$ by adding  (\ref{20221106UVtoA}) and 
(\ref{20221106UVtoB}) or subtracting one from another:
\begin{align}
	U=2\Ree(A+B), \ \ W=2 \Ree(A-B),\ \ V=2 \Imm(A-B). \label{20221106A32}
\end{align}
Then plug (\ref{20221106A30}) and (\ref{20221106A31}) into (\ref{20221106A32}), we 
get 
\begin{align}
	\left(
	\begin{array}{cc}
		U&V\\ V^T& W
	\end{array}
	\right)
	=2
		\left(
	\begin{array}{cc}
		R_1& R_2\\ -R_2& R_1
	\end{array}
	\right)
		\left(
	\begin{array}{cc}
		I+\Lambda&\\& I-\Lambda
	\end{array}
	\right)
		\left(
	\begin{array}{cc}
		R_1^T&-R_2^T\\ R_2^T& R_1^T
	\end{array}
	\right).
\end{align}
Therefore, $\Lambda^2<1$  implies (\ref{eqt:20221106UVW>0}). 
Similar to (\ref{A18_1120}), we have 
\begin{align}
	B \Abarinverse\  \Bbar \Ainverse=R\Lambda^2 R^{-1} .
\end{align}
So $B \Abarinverse\  \Bbar \Ainverse<1$ implies $\Lambda^2<1$.  
 \end{proof}
Lemma \ref{lemma:equivalent_definition_S_Omega_Convexity}
 follows immediately, by letting $A=(\varphi_{\alpha\betabar})$ and $B=(\varphi_{{\alpha\beta}})+S.$
 
The following linear algebra lemma is essential  for the equivalence between modulus of convexity and degree of convexity.
\begin{lemma}
	\label{lemma:Linear_Algebra_Lemma for Equivalence Between Module of Convexity and Degree of Convexity}
	Suppose $A, G$ are positive definite $n\times n$ Hermitian matrices, $B$	is a complex valued symmetric matrix and $\mu$ is a positive constant. Then  
	\begin{align}
		A >\mu G \text{,\ \   and\ \ } B\overline{(A-\mu G)^{-1}} \ \overline{B} (A-\mu G)^{-1}<1 \label{condition:AbiggermuG}
	\end{align} 
if and only if
\begin{align}
	A>0,
	\text{\ and\ } (B-\Theta)\overline{A^{-1}} \ \overline{B-\Theta} A^{-1}<1,
	 \text{\ for any symmetric $\Theta$, with  } \Theta \overline{G^{-1}}\ \Thetabar G^{-1}\leq \mu^2. \label{condition:A>0BThetaA4}
\end{align}  
\end{lemma}

\begin{proof}[Proof of Lemma \ref{lemma:Linear_Algebra_Lemma for Equivalence Between Module of Convexity and Degree of Convexity}]
	We first provide a proof in the $n=1$ case. The lemma in this case is very intuitive and it best demonstrates the idea.  In this case,  $B, \Theta$ are complex numbers, $A, G$ are positive numbers, and, without loss of generality, we can assume $G=1$. As illustrated by Figure \ref{fig:Cone_Lemma_for_Metric_Lower_bound}, 
	the set 
	\begin{align}
		\{(B, A)\big|A>0, |B-\Theta|<A\}
	\end{align}
	is a cone with the corner at $(\Theta, 0)$. 
	\begin{figure}[h]
		\centering
		\includegraphics[height=5.5cm]{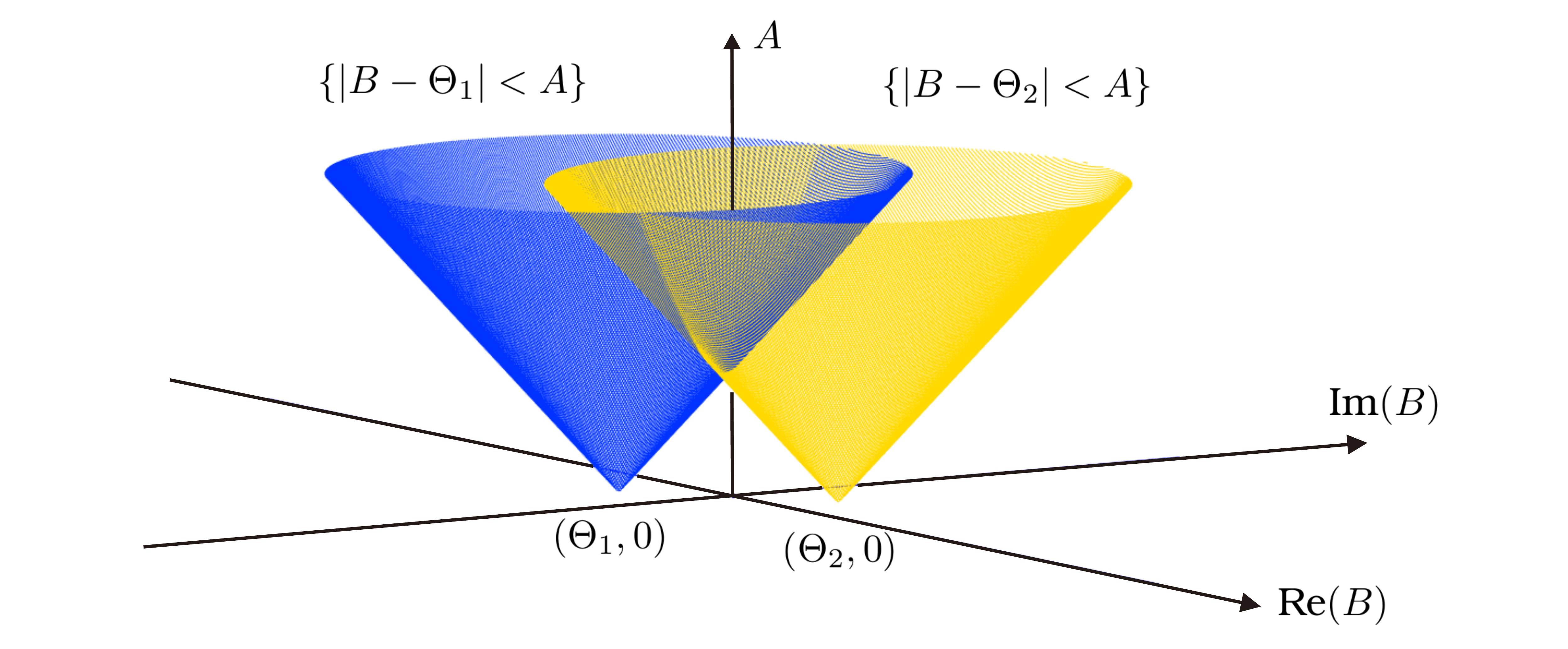}
		\caption{Metric Lower Bound Estimate}
		\label{fig:Cone_Lemma_for_Metric_Lower_bound}
	\end{figure}
	The condition that, for all $\Theta$ with 
	$|\Theta|\leq \mu$, 
	\begin{align}
		(B, A)\subset \{(B, A)\big|A>0, |B-\Theta|<A\}
	\end{align}
	is equivalent to 
	\begin{align}
		(B, A)\subset \bigcap_{|\Theta|\leq \mu}\{(B, A)\big|A>0, 
		|B-\Theta|<A\}. \label{1116A38}
	\end{align}
	The elementary geometry tells us that
	\begin{align}
		\bigcap_{|\Theta|\leq \mu}\{(B, A)\big|A>0, 
		|B-\Theta|<A\}=\{A> \mu +|B|\}.
	\end{align}
	Therefore, (\ref{1116A38}) is equivalent to  
	\begin{align}
	A>\mu\ \ \ \ \ \  \text{and}\ \ \ \ \ \  \left|\frac{B}{A-\mu}\right|<1,
	\end{align} which is condition (\ref{condition:AbiggermuG}).

In general dimension, we need to construct quadratic polynomials from $A, B, G$, similar to the proof of Lemma \ref{lemma:equivalent_definition_S_Omega_Convexity}.
Given an Hermitian matrix $H$ and a symmetric matrix $S$, let
\begin{align}
	K_S^H({\bf z})= H_{{\alpha\betabar}}z^\alpha z^{\betabar}+\Ree(S_{\alpha\beta} z^\alpha z^\beta).
\end{align}

According to Lemma \ref{lemma:equivalent_definition_S_Omega_Convexity}, $A, B, G$ and $\mu$ satisfying (\ref{condition:AbiggermuG}) is equivalent to 
that 
\begin{align}
	K_B^{A-\mu G} \text{\ is a strictly convex function on $\EC^n$};
\end{align}
$A, B, G$ and $\mu$ satisfying (\ref{condition:A>0BThetaA4}) is equivalent to 
that 
\begin{align}
	K_{B-\Theta}^{A} \text{\ is a strictly convex function on $\EC^n$, for any symmetric $\Theta$, with $\Theta\overline{G^{-1}}\ \overline\Theta G^{-1}\leq \mu^2$}.
\end{align}
Because $K_{B-\Theta}^A$ and $K_B^{A-\mu G}$ are both quadratic polynomials, they are strictly convex if and only if they are positive on $\EC^n-\{0\}$. So we need to show 
\begin{align}
	K_B^{A-\mu G}>0, \text{\ on \ } \EC^n\backslash\{0\}
\end{align}
if and only if 
\begin{align}
	K_{B-\Theta}^{A}>0, \text{\ on \ } \EC^n\backslash\{0\}\text{, for any symmetric $\Theta$, with $\Theta\overline{G^{-1}}\ \overline\Theta G^{-1}\leq \mu^2$}. 
\end{align}
Equivalently, we need to show
\begin{align}
	K_B^{A}({\bf z})>\mu G_{\alpha\betabar} z^\alpha z^\betabar, \text{\ \  on \ \ } \EC^n\backslash\{0\}
\end{align}
if and only if 
\begin{align}
	K_{B}^{A}({\bf z})>\Ree(\Theta_{\alpha\beta} z^\alpha z^\beta), \text{\ on \ } \EC^n\backslash\{0\}\text{, for any symmetric $\Theta$, with $\Theta\overline{G^{-1}}\ \overline\Theta G^{-1}\leq \mu^2$}. 
\end{align}
This is valid because
\begin{align}
	\sup_{\Theta \text{ symmetric, } \Theta\overline{G^{-1}}\ \overline\Theta G^{-1}\leq \mu^2}
	 \Ree(\Theta_{\alpha\beta} z^\alpha z^\beta)= 
	\mu G_{\alpha\betabar} z^\alpha z^\betabar.
	\label{1116A48}
\end{align}
To prove (\ref{1116A48}), we find $P$, so that $P GP^\ast=I$, and we let \begin{align}
	\frac{P \Theta P^T}{\mu}=S.
\end{align} Then we find (\ref{1116A48}) is equivalent to 
\begin{align}
	\sup_{S \text{ symmetric, } S \overline S \ \leq 1}
	\Ree(S_{\alpha\beta} z^\alpha z^\beta)= 
\delta_{\alpha\betabar} z^\alpha z^\betabar.
	\label{1116A50}
\end{align}
(\ref{1116A50}) can be easily proved with basic linear algebra.

\end{proof}

Lemma 
\ref{lemma:Equivalence Between Module of Convexity and Degree of Convexity}
follows immediately, by letting $A=(\varphi_{\alpha\betabar})$ and $B=(\varphi_{{\alpha\beta}})+S.$
\subsection{Monotonicity}\label{app:monotone_Lemma}
In this section, we prove the following algebra lemma.
\begin{lemma}[A Monotonicity Lemma]\label{lemma:monotone}
	
	Suppose $A_0$ and $A$ are Hermitian matrices satisfying 
	\begin{align}
		A_0>0\ \ \ \text{and} \ \ \  A_0+A>0, \label{20221106A_0>0_A_0+A>0}
	\end{align}
and $B$ is a symmetric  matrix. Let 
\begin{align}
	K_t=B\overline{(A_0+t A)^{-1}} \ \Bbar (A_0+t A)^{-1}. \label{20221106expressionK_t}
\end{align}
Then $t^{2p}\tr (K_t^p) $ and the maximum eigenvalue of $t^2 K_t$ are both 
non-decreasing functions of $t$. Here $p$ is any positive integer.
\end{lemma}
\begin{proof}
	First of all, we note that condition (\ref{20221106A_0>0_A_0+A>0}) implies 
	\begin{align}
		A_0+tA>0,         \ \ \ \ \text{ for any } t\in (0,1).
	\end{align} So, in (\ref{20221106expressionK_t}),  ${(A_0+t A)^{-1}}$ exists.
In the following we compute 
\begin{align}
	\frac{d}{dt}\left[t^{2p}\tr(K_t^p)\right] 
\end{align}
and show it's non-negative. We need to simultaneously diagonalize $A$ and $ A_0+tA$ 
to simplify the computation.

For $t_0\in[0,1]$, find $P$ so that 
\begin{align}
	PAP^\ast=\Lambda, \ \ \ \ P(A_0+t_0 A)P^\ast=I.  \label{20221106A38}
\end{align}
Here $\Lambda=\diag(\lambda_1, \ ...\ ,\lambda_n)$, with $\lambda_\alpha\in\ER$.
Let
\begin{align}
H=PBP^T,   \label{20221106A39}
\end{align} then $H$ is a symmetric matrix. 
Plug (\ref{20221106A38}) and (\ref{20221106A39}) into 
(\ref{20221106expressionK_t}) to simplify the expression of $K_{t_0}$. We get, at $t=t_0$, 
\begin{align}
	K_{t_0}=P^{-1}HH^\ast P, 
\end{align}
and so
\begin{align}
	P K_{t_0}^{p-1} P^{-1}=(HH^\ast)^{p-1}. \label{20221106expressionHHbar}
\end{align}
Then we compute the derivative,
\begin{align}
	&\frac{d}{dt}\left[\tr(K_t^p) t^{2p}\right]\\
	=&\frac{d}{dt}\left[t^{2p}\tr\left(B\overline{(A_0+t A)^{-1}} \ \Bbar (A_0+t 
	A)^{-1}\right)^p\right]\\
	=&2p\cdot t^{2p-1}\tr(K_t^p)-p\cdot t^{2p} \tr \left[B\overline{(A_0+t A)^{-1}}\  \Abar\ 
	\overline{(A_0+t A)^{-1}} \ \Bbar (A_0+t	A)^{-1} K_t^{p-1}\right]\\
	&\ \ \ \ \ \ \ \ \ \ \ \ \ \ \ \ \  \ -p\cdot t^{2p} \tr \left[B\overline{(A_0+t A)^{-1}}\  \Bbar\ 
	{(A_0+t A)^{-1}} \ A (A_0+t	A)^{-1} K_t^{p-1}\right].
\end{align}
Plug (\ref{20221106A38}) (\ref{20221106A39}) and 
 (\ref{20221106expressionHHbar}) into the 
expression above, we get, at $t=t_0$,
\begin{align}
	\frac{d}{dt}\left[t^{2p}\tr(K_t^p)\right] = p\cdot t_0^{2p-1}
	\left(2\cdot \tr\left[ (HH^\ast)^p\right]-t_0 \cdot\tr\left[(H^\ast H)^p\Lambda 
	\right]-t_0 \cdot
	\tr\left[(H H^\ast)^p\Lambda \right]\right).
\end{align}  Using the fact that $H$ is symmetric, we know
\begin{align}
	\tr\left[(H^\ast H)^p\Lambda \right]=\tr\left[ \Lambda(HH^\ast)^p\right]=\tr\left[ 
	(HH^\ast)^p\Lambda\right].
\end{align}
So
\begin{align}
	\frac{d}{dt}\left[t^{2p}\tr(K_t^p)\right] &= 2p\cdot t_0^{2p-1}
	\left(\tr\left[ (HH^\ast)^p\right]-t_0 \cdot
	\tr\left[(H H^\ast)^p\Lambda \right]\right)\\
	& = 2p\cdot t_0^{2p-1}
	\left(\tr\left[ (HH^\ast)^p(I-t_0\Lambda )\right]\right).\label{20221106A49}
\end{align}  
It's obvious that $H H^\ast$ is semi-positive definite, and, according to 
 (\ref{20221106A38}), 
 \begin{align}
 	I-t_0\Lambda=PA_0P^\ast>0.\label{20221106A50}
 \end{align}
Therefore, we get
\begin{align}
	\frac{d}{dt}\left[t^{2p}\tr(K_t^p)\right]\geq 0.
\end{align}
\end{proof}


\subsection{Concavity}\label{app:Concavity_lemma}
In this appendix, we show the operator (\ref{concave_Function}) is concave. 
Suppose $A$ is an $(n+1)\times (n+1)$ positive definite Hermitian matrix and  $\MG$ is an $n\times n$ positive definite Hermitian matrix. Denote the lower right $n\times n$ block of $A$ by $\mathcal{A}$. Let
\begin{align}
	F_1(A)=\log\left(A_{0\overline{0}}-A_{0\betabar} \mathcal{A}^{\alpha\betabar} A_{\alpha \overline{0}}\right)
\end{align}
and
\begin{align}
	F_2(A)=-\log\left( \MG_{\alpha\betabar}\mathcal{A}^{\alpha\betabar} \right).
\end{align}
We will prove 
\begin{lemma}
	$F_1$ is a concave function of $A$ in the space of positive definite $(n+1)\times (n+1)$ Hermitian matrices;  $F_2$ is a concave function of $\mathcal{A}$ in the space of positive definite $n\times n$ Hermitian matrices.
\end{lemma}
\begin{proof}

For the concavity of $F_1$, actually, we can show 
\begin{align}
f_1(A)=A_{0\overline{0}}-A_{0\betabar} \mathcal{A}^{\alpha\betabar} A_{\alpha \overline{0}}
\end{align}
is a concave function of $A$. 
Let $H$ be an $(n+1)\times (n+1)$ Hermitian matrix. Similar to $A$, denote the lower-right block of $H$ by $\MH$. Let 
\begin{align}
	q(t)=f_1(A+tH),  \label{expression_a77}
\end{align} for $t$ close to $0$.
We will show that $q''(0)\leq0.$ 
To simplify the computation, we diagonalize $\mathcal{A}$ and $\MH$ simultaneously. Find an $n\times n$ matrix $\MP$, so that
\begin{align}
	\MP \mathcal{A}\MP^\ast=I; \ \ \ \ \ \ \MP\MH\MP^\ast=\Lambda=\diag(\lambda_1,\ ...\ , \lambda_n).
\end{align}
Let
\begin{align}
	P=
	\left(
	\begin{array}{cc}
		1&0\\0&\MP
	\end{array}
	\right).
\end{align}
Note that another expression for $f_1$ is
\begin{align}
	f_1(A)=\frac{\det A}{\det{\mathcal{A}}}.
\end{align}
So
\begin{align}
	q(t)=\frac{\det (A+tH)}{\det (\mathcal{A}+t\MH)}
		=\frac{\det[P({{A}}+tH)P^\ast]}{\det[\MP({\mathcal{A}}+t\MH)\MP^\ast]}.
\end{align}
Denote that
\begin{align}
	PAP^\ast=\left(
	\begin{array}{cccc}
		a_{0\overline{0}}& \cdots & u_{0\betabar}& \cdots \\ \vdots& & &\\
		u_{\alpha \overline{0}}&  & I\\ \vdots&
	\end{array}
	\right), \ \ \ \ \
		PHP^\ast=\left(
	\begin{array}{cccc}
		h_{0\overline{0}}& \cdots & v_{0\betabar}& \cdots \\ \vdots& & &\\
		v_{\alpha \overline{0}}&  & \Lambda\\ \vdots&
	\end{array}
	\right).
\end{align}
With these simplifications,
\begin{align}
	q(t)=a_{0\overline{0}}+th_{0\overline{0}}-
	\sum_{\alpha}\frac{(u_{0\alphabar}+tv_{0\alphabar})( u_{\alpha \overline{0}}+tv_{\alpha \overline{0}})}{1+t \lambda_\alpha}.
\end{align}
Straightforward computation gives
\begin{align}
	q''(0)=-\sum_\alpha(u_{0\alphabar}\lambda_\alpha-v_{0\alphabar})( u_{\alpha \overline{0}}\lambda_\alpha-v_{\alpha \overline{0}})\leq 0.
\end{align}
Therefore, $f_1$ is concave, and, consequently, $F_1=\log(f_1)$ is concave.

For the concavity of $F_2$, actually, we can show
\begin{align}
	f_2=\frac{1}{\tr(\MG\mathcal{A}^{-1})}
\end{align}
is a concave function of $\mathcal{A}$. This is a simple consequence of the well-known fact that
\begin{align}
\frac{1}{\tr(\mathcal{A}^{-1})}=\frac{\det(\mathcal{A})}{\sigma_{n-1}(\mathcal{A})}
\end{align}
is a concave function of $\mathcal{A}$. Therefore, $F_2=\log(f_2)$ is concave.

\end{proof}

\section*{Acknowledgement}
This work is supported by the National Natural Science Foundation of China (No. 12288201) and the Project of Stable Support for Youth Team in Basic Research Field, CAS, (No. YSBR-001). The author would like to thank   Li Chen, Xiuxiong Chen, Jianchun Chu, Jiyuan Han, Laszlo Lempert, Long Li, Yu Li, Guohuan Qiu, Zaijiu Shang, Li Sheng, Bing Wang, Youde Wang, Bin Xu for very helpful discussions.


{\flushleft Jingchen Hu}\\
Loo-Keng Hua Center for Mathematical Sciences\\
Email:
JINGCHENHOO@GMAIL.COM

\end{document}